\documentclass[10pt]{amsart}
\usepackage{amscd,amssymb}
\usepackage{enumerate}
\usepackage[all]{xypic}
\usepackage{multirow}
\usepackage{hhline}
\usepackage[mathscr]{eucal}
\usepackage{verbatim}

\theoremstyle{plain}
\newtheorem{Thm}[equation]{Theorem}
\newtheorem*{Thm*}{Theorem}

\newtheorem{Prop}[equation]{Proposition}
\newtheorem{Lem}[equation]{Lemma}
\newtheorem{Rmk}[equation]{Remark}
\newtheorem{Def}[equation]{Definition}
\newtheorem{Conj}[equation]{Conjecture}

\numberwithin{equation}{section}

\newcommand{\Hom}{\operatorname{Hom}}
\newcommand{\Gal}{\operatorname{Gal}}
\newcommand{\GL}{\operatorname{GL}}
\newcommand{\PGL}{\operatorname{PGL}}
\newcommand{\Aut}{\operatorname{Aut}}

\newcommand{\Sp}{\operatorname{Sp}}
\newcommand{\Spt}{\widetilde{\operatorname{Sp}}}

\newcommand{\OO}{\operatorname{O}}

\newcommand{\Ps}{\operatorname{Ps}}
\newcommand{\supp}{\operatorname{supp}}

\newcommand{\Ut}{\widetilde{U}}

\renewcommand{\O}{\mathcal{O}}

\renewcommand{\OE}{\mathcal{O}_E}

\renewcommand{\P}{\mathcal{P}}

\newcommand{\PE}{\mathcal{P}_E}

\newcommand{\tr}{\operatorname{tr}}

\newcommand{\Span}{\operatorname{span}}
\newcommand{\ord}{\operatorname{ord}}
\newcommand{\End}{\operatorname{End}}

\newcommand{\C}{\mathbb C}

\newcommand{\Z}{\mathbb{Z}}

\newcommand{\V}{\mathcal{V}}
\newcommand{\U}{\mathfrak{u}}
\newcommand{\He}{\mathcal{H}}

\newcommand{\Id}{\operatorname{Id}}

\newcommand{\ie}{{\em i.e. }}

\newcommand{\la}{\langle}
\newcommand{\ra}{\rangle}

\title[the lattice model and the Howe duality conjecture]{On the
  lattice model of the Weil representation and the Howe duality
  conjecture}
\author{Shuichiro Takeda}
\address{Department of Mathematics, University of Missouri
202 Math Science Bldg 65211}

\begin{document}

\maketitle

\begin{abstract}
The lattice model of the Weil representation over non-archimedean
local field $F$ of odd residual characteristic has been known for
decades, and is used to prove the Howe duality conjecture for unramified dual
pairs when the residue characteristic of $F$ is odd. In this paper, we
will modify the lattice model of the Weil representation so that it is
defined independently of the residue characteristic. Although
to define the lattice model alone is not enough to prove the Howe
duality conjecture for even residual characteristic, we will propose a
couple of conjectural lemmas which imply
the Howe duality conjecture for unramified dual pairs for even
residual characteristic. Also we will
give a proof of those lemmas for certain cases, which allow us to
prove (a version of) the Howe duality conjecture for even residual
characteristic for a certain class of representations for the dual
pair $({\rm O}(2n), {\rm Sp}(2n))$, where ${\rm O}(2n)$ is unramified.
We hope this paper serves as a first step toward a proof of the Howe
duality conjecture for even residual characteristic.
\end{abstract}


\section{\bf Introduction}


Let $F$ be a non-archimedean local field of characteristic 0 and $W$
be a symplectic space over $F$ of dimension $2n$. For an additive
character $\psi$ on $F$, we let $\omega_{\psi}$ be the Weil
representation of the metaplectic cover $\Spt(W)$ of $\Sp(W)$. Let $E$
be either $F$ or a quadratic
extension of $F$. For $i=1, 2$, let $(\V_i, \la-,-\ra_i)$ be an
$\epsilon_i$-Hermitian space over $E$ where
$\epsilon_1\in\{\pm1\}$ and $\epsilon_1\epsilon_2=-1$, and let
$U(\V_i)$ be its isometry group. Assume the
pair $(U(\V_1), U(\V_2))$ forms an irreducible dual reductive pair in
$\Spt(W)$, so that $U(\V_1)\cdot U(\V_1)$ is a subgroup of
$\Sp(W)$. We call the restriction of $\omega_\psi$ to the preimage of
$U(\V_1)\cdot U(V_2)$ in $\Spt(W)$ also $\omega_\psi$. In this
introduction just for notational convenience we assume that both
$U(\V_1)$ and $U(V_2)$ split in $\Spt(W)$, so we may view
$\omega_\psi$ as a representation of $U(\V_1)\cdot U(\V_1)$, or even
as a representation of $U(\V_1)\times U(\V_2)$ via the multiplication
map $U(\V_1)\times U(\V_2)\rightarrow U(\V_1)\cdot U(\V_2)$.

For an irreducible admissible representation $\pi_1$ of $U(\V_1)$, the
maximum $\pi$-isotypic quotient of $\omega_\psi$ (as a representation
of $U(\V_1)\times U(V_2)$) has the form
\[
\pi\otimes\Theta_\psi(\pi)
\]
for some (possibly zero) smooth representation $\Theta_\psi(\pi)$ of
$U(\V_2)$. It is known that $\Theta_\psi(\pi)$ is of finite length and
hence is admissible. We let $\theta_\psi(\pi)$ be the maximal
semisimple quotient of $\Theta_\psi(\pi)$. It has been conjectured by
Howe that 
\begin{itemize}
\item $\theta_\psi(\pi)$ is irreducible whenever $\Theta_\psi(\pi)$ is
  non-zero.
\item the map $\pi\mapsto\theta_\psi(\pi)$ is injective on its domain.
\end{itemize}
This conjecture has been known as the Howe duality conjecture, and
proven by Howe and Waldspurger when the residue characteristic of
$F$ is odd more than two decades ago (\cite{MVW}, \cite{H90},
\cite{Wa}). 

The case for even residual characteristic is still widely open in
general. To the best of our knowledge, the only general result for
even residual characteristic is the quarter century old result by Kudla (\cite{Kudla})
in which he shows if $\pi$ is supercuspidal,  then $\Theta_\psi(\pi)$
is always irreducible. Since then, however, it seems no progress has been made,
possibly with the exceptions of the recently result by Li-Sun-Tian
(\cite{LST}) which shows that $\theta_\psi(\pi)$ is multiplicity free. Also
when the ranks for the groups are very small like 2 or 3 at most, one
can check the Howe duality by hand. 

The proof of the Howe duality conjecture for odd residual
characteristic is reproduced in detail by Waldspurger in
\cite{Wa}. The proof requires what is known as the generalized lattice
model of the Weil representation and is highly complex. However when
the dual pair $(U(\V_1), U(\V_2))$ is ``unramified'' in the sense of
\cite{MVW}, which is the same as saying both of the groups $U(\V_1)$
and $U(\V_2)$ are unramified \ie split over unramified extension of
$F$, one only needs the lattice
model of the Weil representation, which is much simpler than the
generalized lattice model, and hence the proof becomes significantly
simpler. This proof is reproduced in Chapter 5 of
\cite{MVW}. Also another version of the proof is given in \cite{H90}.

All of those proofs of the Howe duality conjecture for odd residual
characteristic require the (generalized) lattice model of the Weil
representation. One of the crucial obstructions to apply those proofs
to the case of even residual characteristic is the
unavailability of such model. 

In this paper, we modify the known lattice model so that it can be
defined even when the residue characteristic of $F$ is even. As we
will see, however, this is not enough to extend the proof in
\cite[Ch. 5] {MVW} to the case of even residual
characteristic, and various technical difficulties arise if one
simply tries to apply the arguments in \cite[Ch. 5]{MVW} to the case of even residual
characteristic. In particular, one has to prove a  couple of lemmas, which we
call ``the first and second key lemmas on the lattice model'', which are the
analogues of Theorem I.4 and Proposition I.5 in Chapter 5 of
\cite[p. 103]{MVW}, respectively.

Let us be more specific. For an unramified dual pair $(U(\V_1), U(\V_2))$,
there exists a self-dual lattice $L_i$ of $\V_i$ for $i=1, 2$, so that
$A:=L_1\otimes_{\OE} L_2$ is a self-dual lattice of
$W:=\V_1\otimes_E\V_2$ with respect to the additive character $\psi$. First we show 
\begin{Thm}
There exists the lattice model $(\omega_\psi, S_A)$ of the Weil
representation $\omega_\psi$ even when the residue characteristic of
$F$ is even.
\end{Thm}

Here the space $S:=S_A$ of the lattice model is a certain set of smooth
compactly supported functions $f:W\times F\rightarrow\C$ on $W\times F$. For each
sublattice $L\subseteq L_1$ of the self-dual lattice $L_1$, we define
$S_L$ to be the subspace of
$S_A$ consisting of functions whose support is in
$(L^\perp\otimes_{\OE}L_2)\times F$, where $L^\perp$ is the dual lattice
of $L$.

The first key lemma roughly says the following:
For any sublattice $L\subset L_1$, we have the equality
\[
\omega_\psi(\He_2)S_L=S^{J_1(L)}
\]
where $\He_2$ is the spherical Hecke algebra of $U(\V_2)$, and $S^{J_1(L)}$ is the set of
functions invariant by the subgroup $J_1(L)$ of $U(\V_1)$ which
is the kernel of the reduction map
$U(L^\perp)\rightarrow\Aut(L^\perp/L)$, where $L^\perp$ is the dual
lattice of $L$. (Strictly speaking we need
to modify $J_1(L)$ by a certain subgroup $J_1(L)^\circ$, which will be
defined later in the paper.) The same should hold by switching
the roles of $\V_1$ and $\V_2$. 

The second key lemma is even more of technical nature: First we define
$H_1(L)$ to be the kernel of the reduction map $U(L_1)\rightarrow
\Aut(L_1/L)$. Then it turns out that for each $w\in W$, there is a
character $\psi_1^w$ which naturally arrises in the theory of lattice
model. Then the second key lemma says if $\psi_1^w=\psi^{w'}$ for
$w,w'\in W$ and both $w$ and $w'$ satisfy a certainly maximality
condition, then $w$ and $w'$ have to be in the same orbit under the
action of the maximal open compact subgroup of $U(\V_1)$ on $W/A$. And
the same should hold by switching the roles of $\V_1$ and $\V_2$.

With those two conjectural lemmas, we can show the following version
of the Howe duality principle: 
\begin{Thm}
Modulo the above two conjectural lemmas, the Howe duality conjecture holds
independently of the residue characteristic, in the sense that if
$\pi$ is an irreducible admissible representation of $U(\V_1)$ and
$\Theta_\psi(\pi)\neq 0$, then $\Theta_\psi(\pi)$ has a unique
irreducible non-zero quotient.
\end{Thm}

As we already mentioned, we are not able to prove the two key lemmas
in full generality. However, if $L$ is of the form $L=\varpi^k L_1$,
where $\varpi$ is a uniformizer of $F$ and $k$ is an integer with
$k\geq 1+e$, where $e$ is the ramification index of $2$ in $F$, then
the two key lemmas can be proven to the extent necessary to prove

\begin{Thm}[Main Theorem]
Let $\pi$ be an irreducible admissible representation of ${\rm
  O}(\V)$, where ${\rm O}(\V)$ is quasi-split and split over an
unramified extension and $\dim\V=2n$. Let $L_1\subseteq\V$ be a
self-dual lattice. Assume 
\begin{itemize}
\item $\pi^{J_1(L)}= 0$ for all $L$ with $L\supsetneq \varpi^{k}L_1$;
\item $\pi^{J_1(L)}\neq 0$ for $L=\varpi^{k}L_1$ for some $k$ with $k>1+e$.
\end{itemize}
Then for the dual pair $({\rm O}(2n), {\rm Sp}(2n))$, if
$\Theta_\psi(\pi)\neq 0$, it has a unique non-zero irreducible quotient.
\end{Thm}

In the above theorem we need to assume that the symplectic group has
the same rank as the orthogonal group. All the conditions we need to
impose on this theorem are all of technical nature.

Also we consider the lifting to smaller rank symplectic
groups and prove

\begin{Thm}
Let $\pi$ be an irreducible admissible representation of ${\rm
  O}(\V)$, where ${\rm O}(\V)$ is
quasi-split and split over an unramified extension, with $\dim\V=2m$
or $\dim\V=2m+1$. Assume $\pi$ is
such that
\[
V_\pi^{J_1(L)}=0
\]
 for all $L\supsetneq 2\varpi L_1=\varpi^{1+e}L_1$ for a self-dual
 lattice $L_1\subseteq\V$. Then for
the dual pair $({\rm O}(\V), {\rm Sp}(2n))$, (or  $({\rm O}(\V), {\rm
  \widetilde{Sp}}(2n))$) with $m>n$, we always have
$\Theta_\psi(\pi)=0$.
\end{Thm}

The main structure of this paper is the following: In the next section (Section 2),
we will go over
the formulation  of the Heisenberg group and the Weil representation. Our
formulation differs from the modern convention, but closely follows
the one in the original paper by Weil (\cite{Weil}). In Section 3, we
will define the lattice model of the Weil representation, which works
independently of the residual characteristic, and make explicit the action of
the metaplectic group on this model. In Section 4, we extend some of
the lemmas about lattices proven in \cite[Ch. 5, II]{MVW} to the case
of even residual characteristic, and in Section 5 we will go over the
notion of unramified dual pair. In Section 6, we formulate the two
(conjectural) key lemmas. In Section 7 and 8 we prove the first and
second key lemmas for the special type of lattices mentioned
above namely those $L$ with $L\subseteq 2\varpi L_1=\varpi^{1+e}L_1$. Then finally in
Section 9, we give our proof of the main
theorem, and in Section 10 we will prove the last theorem mentioned above.

\quad\\
\begin{center} {\bf Notations} \end{center}
Throughout the paper, $F$ will be a non-archimedean local field of
characteristic $0$, and $E$ will be either $F$ or an unramified quadratic extension of
$F$.  We let $\O$ (resp. $\OE$) be the ring of integers of $F$
(resp. $E$). We let $\varpi$ be a chosen uniformizer of $F$, and
choose our uniformizer of $E$ to be $\varpi$ as well.
Also we write $\P^n=\varpi^n\O$ and
$\PE^n=\varpi^n\OE$ fo each integer $n$. We fix an additive
character $\psi$ of $F$, and $r$ be the exponential conductor of $\psi
$ so that $\psi$ is trivial on $\P^r$. Also we let 
\[
e=\ord_F(2),
\]
so $2=\varpi^e\times\text{unit}$, and in particular if the residue
characteristic of $F$ is even, $e$ is the ramification index of $2$ in
$F$.

For each $c\in E$, we denote $\bar{c}=c$ if $E=F$ and
$\bar{c}=\tau(c)$ if $E\neq F$ where $\tau$ is the non-trivial element in
$\Gal(E/F)$.

For $\epsilon\in\{\pm 1\}$, by an $\epsilon$-Hermitian space $(\V,
\la-,-\ra)$ over $E$,
we mean a finite dimensional vector space $\V$ over $E$ equipped with
a map $\la-,-\ra:\V\times\V\rightarrow E$ which is linear on the
first argument and antilinear on the second with the property that
$\la v_1, v_2\ra=\epsilon\overline{\la v_2, v_1\ra}$. We always assume
our $\epsilon$-Hermitian space is nondegenerate. We let $U(\V)$ be the
group of isometries of $(\V, \la-,-\ra)$. By a lattice $L$ of $\V$, we
mean a free $\OE$-module $L$ whose rank is equal to $\dim_E\V$.
\quad\\

\begin{center} {\bf Acknowledgements} \end{center}

The author would like to thanks Roger Howe for sending him the
unpublished manuscripts on the theta correspondence (\cite{H_un}). He
is partially supported by NSF grant DMS-1215419. Also a part of the
paper was written when the author was visiting the Fields
Institute. He would like to thank their hospitality. 


\section{\bf The Heisenberg group and the Weil representation}


Let $(W,\la-,-\ra)$
be a symplectic space over $F$ of dimension $2n$. In this section, we
define the Heisenberg group and the Weil representation for $W$, but
in order to construct the lattice model of the Weil representation
that works for the case of even residual characteristic, we need to
adapt a different convention. The Heisenberg group $H(W)$
associated with $W$ is usually defined to be $H(W)=W\times F$ as a
set with the group structure given by
\[
(w_1,z_1)\cdot (w_2,z_2)=(w_1+w_2, z_1+z_2+\frac{1}{2}\la w_1,w_2\ra)
\]
for $(w_i, z_i)\in H(W)$, and then one can see that each $g\in\Sp(W)$ acts on $H(W)$ by
$g\cdot(w,z)=(gw,z)$. However the
$\frac{1}{2}$ appearing here makes it impossible to define the lattice model of the
Weil representation when the residue characteristic of $F$ is even. To
get around it, we will define the Heisenberg group differently by
following the original formulation by Weil (\cite{Weil}).

First we need to fix a polarization 
\[
W=W^+\oplus W^-,
\]
and for each element $w\in W$, we write $w=w^++w^-$ where $w^+\in W^+$
and $w^-\in W^-$. With respect to this polarization, we define a
bilinear form
\[
\beta:W\times W\rightarrow F,\quad (w_1, w_2)\mapsto \langle
w_1^+, w_2^-\rangle.
\]
Note that $\beta$ is indeed bilinear and $\beta(w_1,w_2)=\la w_1^+,
w_2\ra=\la w_1, w_2^-\ra$ but in general
$\beta(w_1,w_2)\neq-\beta(w_2,w_1)$, but instead we have
\[
\beta(w_1,w_2)-\beta(w_2,w_1)=\la w_1,w_2\ra.
\]
For each
$\beta$, we define the Heisenberg group $H_\beta(W)$ to be the group
with underlying set
\[
H_\beta(W)=W\times F
\]
where the group operation is given by
\[
(w_1,z_1)\cdot (w_2,z_2)=(w_1+w_2, z_1+z_2+\beta(w_1,w_2)).
\]
One can check that the center of $H_\beta(W)$ is $\{(0,z):z\in
F\}$. 

\begin{Prop}
For any $\beta$, $H_\beta(W)$ is isomorphic to the usual Heisenberg
group $H(W)$ defined by $\frac{1}{2}\la-,-\ra$.
\end{Prop}
\begin{proof}
By direct computation, one can check that the map
$H_\beta(W)\rightarrow H(W)$ defined by $(w,z)\mapsto (w,
z-\frac{1}{2}\la w^+, w^-\ra )$ is an isomorphism.
\end{proof}

\begin{Prop}
All the $H_\beta(W)$ are isomorphic to each other.
\end{Prop}
\begin{proof}
Of course, this immediately follows from the previous proposition
because $H(W)$ is independent of $\beta$. But one can also construct
an explicit isomorphism as follows: 
Let $W=X\oplus Y$ be another polarization. Define $\beta':
W\rightarrow F$ analogously with respect to this polarization. By
Witt's extension theorem, there exists $g\in\Sp(W)$ such that
$g(W^+)=X$ and $g(W^-)=Y$. Then one can see that the map
$H_\beta(W)\rightarrow H_{\beta'}(W)$ defined by $(w,
z)\mapsto (g(w), z)$ is a isomorphism.
\end{proof}

If we define our Heisenberg group in this way, however, we no longer
have $\beta(gw_1, gw_2)=\beta(w_1, w_2)$ for every $g\in\Sp(W)$, and hence
$\Sp(W)$ does not act in any obvious way. Namely the discrepancy
$\beta(gw_1,gw_2)-\beta(w_1,w_2)$ has to be taken care of. For this
purpose, let us define, for each $g\in\Sp(W)$, $\Sigma_g$ to be the set of all
continuous functions $\alpha:W\rightarrow F$ such that
\[
\alpha(w_1+w_2)-\alpha(w_1)-\alpha(w_2)=\beta(gw_1,gw_2)-\beta(w_1,w_2).
\]
Such function $\alpha$ is a character of second degree in the
sense of \cite{Weil}.
Following Weil (\cite{Weil}), we define the linear pseudosymplectic group
$\Ps(W)$ by
\[
\Ps(W):=\{(g,\alpha):g\in\Sp(W), \alpha\in\Sigma_g\},
\]
where the group structure is given by
\[
(g_1,\alpha_1)\cdot (g_2,\alpha_2)=(g_1g_2,\;g_2^{-1}\cdot\alpha_1+\alpha_2)
\]
where $g_2^{-1}\cdot\alpha_1$ is defined by
\[
g_2^{-1}\cdot\alpha_1(w)=\alpha_1(g_2w).
\]
Then $\Ps(W)$ acts on $H_\beta(W)$ as
\[
(g,\alpha)\cdot (w, z)=(gw,\; z+\alpha(w)).
\]
One can verify that this is indeed an action.

We have the obvious map $\Ps(W)\rightarrow \Sp(W)$ given by
$(g,\alpha)\mapsto g$. Weil shows that the sequence
\[
0\rightarrow W^\ast\rightarrow \Ps(W)\rightarrow\Sp(W)\rightarrow 0
\]
is exact, where $W^\ast=\Hom_F(W, F)$ (see
\cite[p. 150]{Weil}). Moreover, he shows that the exact
sequence splits by the following lemma.

\begin{Lem}\label{L:splitting}
Let $g\in\Sp(W)$ have the matrix representation
\[
g=\begin{pmatrix}a&b\\c&d\end{pmatrix}
\]
with respect to the polarization $W=W^+\oplus W^-$. Let
\begin{equation}\label{E:alpha_g}
\alpha_g(w^++w^-)=\frac{1}{2}\la aw^+, cw^+\ra+\frac{1}{2}\la
bw^-,dw^-\ra+\la bw^-, cw^+\ra.
\end{equation}
Then $\alpha_g\in\Sigma_g$. Moreover, the map $g\mapsto(g,
\alpha_g)$ gives a homomorphism from $\Sp(W)$ to $\Ps(W)$, namely 
\[
\alpha_{gh}=h^{-1}\cdot\alpha_g+\alpha_h
\quad\text{for all $g,h\in\Sp(W)$}.
\]
\end{Lem}
\begin{proof}
Though this is proven in \cite[Sec 4]{Weil}, it is not so easy to read
it off from there due to the notational discrepancy. Hence we give a
proof here with our notations. First let us show
$\alpha_g\in\Sigma_g$. Let $w_1=w_1^++w_1^-$ and $w_2=w_2^++w_2^-$ be
in $W=W^+\oplus W^-$. First we need to show that $\alpha_g\in\Sigma_g$, namely
\[
\alpha_g(w_1+w_2)-\alpha_g(w_1)-\alpha_g(w_2)-\beta(gw_1,gw_2)+\beta(w_1,w_2)=0.
\]
By keeping in mind
\[
\beta(w_1,w_2)=\la w_1^+,w_2^-\ra=\la gw_1^+,gw_2^-\ra
=\la aw_1^++cw_1^+,bw_2^-+dw_2^-\ra=\la aw_1^+, dw_2^-\ra+\la cw_1^+,bw_2^-\ra,
\]
one can show by direct computations that
\begin{align*}
&\alpha_g(w_1+w_2)-\alpha_g(w_1)-\alpha_g(w_2)-\beta(gw_1,gw_2)+\beta(w_1,w_2)\\
&=-\frac{1}{2}\la aw_1^+, cw_2^+\ra+\frac{1}{2}\la aw_2^+, cw_1^+\ra
-\frac{1}{2}\la bw_1^-, dw_2^-\ra+\frac{1}{2}\la bw_2^-, dw_1^-\ra.
\end{align*}
Here notice that
\begin{align*}
-\frac{1}{2}\la aw_1^+, cw_2^+\ra+\frac{1}{2}\la aw_2^+, cw_1^+\ra
&=\frac{1}{2}\la cw_2^+, aw_1^+\ra+\frac{1}{2}\la aw_2^+, cw_1^+\ra\\
&=\frac{1}{2}\la aw_2^++cw_2^+, aw_1^+\ra+\frac{1}{2}\la aw_2^++cw_2^+, cw_1^+\ra\\
&=\frac{1}{2}\la aw_2^++cw_2^+, aw_1^++cw_1^+\ra\\
&=\frac{1}{2}\la gw_2^+, gw_1^+\ra\\
&=\frac{1}{2}\la w_2^+, w_1^+\ra\\
&=0.
\end{align*}
Similarly one can show that $-\frac{1}{2}\la bw_1^-,
dw_2^-\ra+\frac{1}{2}\la bw_2^-, dw_1^-\ra=0$. Hence we have shown
$\alpha_g\in\Sigma_g$. 

To show that the map $g\mapsto (g,\alpha_g)$ is a group homomorphism
is even a more straight forward computation, though tedious. 
\end{proof}

\begin{Rmk}
Let us mention that in the above lemma, if $b=c=0$, namely $g$ is in
the Siegel Levi, then $\alpha_g=0$.
\end{Rmk}

We would like to describe $(g,\alpha_g)^{-1}$ for each $g\in
\Sp(W)$. For this purpose, define $\alpha^g=-\alpha_g\circ
g^{-1}$. Namely it is the map $\alpha^g:W\rightarrow F$ defined by
\[
\alpha^g(w)=-\alpha_g(g^{-1}w)
\]
for $w\in W$.  Then we have

\begin{Lem}\label{L:alpha^g}
For each $g\in\Sp(W)$,
\begin{enumerate}
\item $\alpha^g\in\Sigma_{g^{-1}}$;
\item $(g^{-1},\alpha^g)=(g,\alpha_g)^{-1}$ in $\Ps(W)$.
\end{enumerate}
\end{Lem}
\begin{proof}
$(1)$ For $w_1, w_2\in W$, we have
\begin{align*}
\alpha^g(w_1+w_2)-\alpha^g(w_1)-\alpha^g(w_2)
&=-\alpha_g(g^{-1}(w_1+w_2))+\alpha_g(g^{-1}w_1)+\alpha_g(g^{-1}w_2)\\
&=-\beta(gg^{-1}w_1, gg^{-1}w_2)+\beta(g^{-1}w_1,g^{-1}w_2)\\
&=\beta(g^{-1}w_1,g^{-1}w_2)-\beta(w_1,w_2).
\end{align*}
Hence $\alpha^g\in\Sigma_{g^{-1}}$.

$(2)$ By part $(1)$, one knows that indeed
$(g^{-1},\alpha^g)\in\Ps(W)$. Now consider
\[
(g,\alpha_g)(g^{-1},\alpha^g)=(1, g\cdot\alpha_g+\alpha^g),
\]
and for all $w\in W$,
\[
(g\cdot\alpha_g+\alpha^g)(w)=\alpha_g(g^{-1}w)-\alpha_g(g^{-1}w)=0,
\]
and so $g\cdot\alpha_g+\alpha^g=0$. Hence
$(g^{-1},\alpha^g)=(g,\alpha_g)^{-1}$.
\end{proof}

\begin{Rmk}\label{R:important}
Throughout this paper, we view $\Sp(W)$ as a subgroup of $\Ps(W)$ via
the splitting $g\mapsto(g,\alpha_g)$, and when we denote an element
$g\in\Sp(W)$ it should be considered as an
abbreviation of $(g,\alpha_g)$. In particular $g^{-1}$ has to be
considered as $(g^{-1}, \alpha^g)$ rather than $(g^{-1},
\alpha_{g^{-1}})$. Note that in general $\alpha_{g^{-1}}\neq\alpha^g$.
\end{Rmk}

Recall 

\begin{Thm}[Stone-Von-Neumann]
For a fixed additive character $\psi$, there is
a unique (up to isomorphism) smooth irreducible representation
$\rho_W$ of
$H_\beta(W)$ such that the element $(0,z)$ in the center acts as
multiplication by $\psi(z)$.
\end{Thm}

Via the splitting $\Sp(W)\rightarrow\Ps(W)$, the symplectic group
$\Sp(W)$ acts on $H_\beta(W)$, namely
\[
g\cdot (w,z)=(gw, z+\alpha_g(w))
\]
for $g\in\Sp(W)$ and $(w,z)\in H_\beta(W)$. This, combined with the
Stone-Von-Neumann theorem, gives rise to the projective
representation $\Sp(W)\rightarrow\PGL(\rho_W)$. This projective
representation defines the
metaplectic cover $\Spt(W)$ which is the subgroup of
$\Sp(W)\times\GL(\rho_W)$ that consists of pairs $(g, M_g)$ where
$g\in\Sp(W)$ and $M_g\in\GL(\rho_W)$ are such that
\[
M_g\circ\rho_W(h)=\rho_W(g\cdot h)\circ M_g
\]
for all $h\in H_\beta(W)$, where $g\cdot h$ is the above mentioned
action of $\Sp(W)$ on $H_\beta(W)$. Note that for each fixed
$g\in\Sp(W)$, the map $h\mapsto \rho_W(g\cdot h)$ defines another irreducible
representation of $H_\beta(W)$, and hence by Stone-Von-Neumann, this
is equivalence to $\rho_W$.
Let $\omega_\psi$ be the representation of
$\Spt(W)$ on the space of $\rho_W$ given by 
\[
\omega_\psi(g, M_g):=M_g.
\]
This representation is called the Weil representation of $\Spt(W)$. Also note that we have
the short exact sequence
\[
1\rightarrow \C^\times \rightarrow\Spt(W)\rightarrow\Sp(W)\rightarrow 1,
\]
where the map $\C^\times \rightarrow\Spt(W)$ is the inclusion
$z\mapsto (1, z\Id)$ and the map $\Spt(W)\rightarrow\Sp(W)$ is the
projection $(g, M_g)\mapsto g$.

\begin{Rmk}
The Weil representation $\omega_\psi$ can be shown to be independent
of the choice of $\beta$. Also the metaplectic group $\Spt(W)$ can be
shown to be independent of $\psi$ and $\beta$.
\end{Rmk}


\section{\bf The Lattice Model of the Weil
  representation}\label{S:lattice_model}


Let $A$ be a lattice of $W$, namely a free $\O$-module of $W$ of rank
equal to $\dim W$. We define the dual $A^\perp$ of $A$ with respect to
an integer $r$ by
\[
A^\perp:=\{w\in W: \la a, w\ra\in\P^r\text{ for all } a\in A\}.
\]
We usually assume $r$ to be the exponential conductor of our fixed
additive character $\psi$, and then the condition $\la a, w\ra\in\P^r$
for all $a\in A$
is equivalent to $\psi(\la a, w\ra)=1$ for all $a\in A$. We say a lattice $A$ is
self-dual if $A=A^\perp$. Note that if
$A$ is self-dual, then $\la a_1, a_2\ra\in\P^r$ for all $a_1, a_2\in
A$. Given a self-dual lattice $A$ (with respect to $r$) one can always choose the
polarization $W=W^+\oplus W^-$ such that we have the decomposition
$A=(A\cap W^+)\oplus (A\cap W^-)$. Conversely for each polarization
$W= W^+\oplus W^-$, one can always find a self-dual lattice $A$ with
respect to $r$ so
that $A=(A\cap W^+)\oplus (A\cap W^-)$. When a self-dual lattice $A$
is decomposed in this way with resect to the polarization $W= W^+\oplus
W^-$, we say that $A$ is compatible with the polarization. If $A$ is
compatible with the polarization defining $\beta$, one can see that
$\beta(a_1, a_2)\in\P^r$ for all $a_1, a_2\in A$.

For a self-dual lattice $A$ compatible with our fixed polarization of $W$, we let 
\[
H_\beta(A):=A\times F\subset H_\beta(W).
\]
Then $H_\beta(A)$ is a subgroup of $H_\beta(W)$. Define the character
\[
\psi_A:H_\beta(A)\rightarrow\C^\times, \quad (a, z)\mapsto \psi(z).
\]
Since $A$ is compatible with the polarization, one can see that this
map is indeed a character. Moreover one can see that $\psi_A$ cannot
be extended to any bigger subgroup than $H_\beta(A)$. Consider the
induced representation
\[
S_A:=ind_{H_\beta(A)}^{H_\beta(W)}\psi_A,
\]
where the induction is compact induction.
Then we have
\begin{Prop}
The above induced representation provides a model of the Heisenberg
representation.
\end{Prop}
\begin{proof}
It is immediate that the central character is $\psi$. Hence it
suffices to show that it is irreducible. The proof is essentially the
same as the usual lattice model (see \cite[p. 29]{MVW}). Yet, since
the proof is not identical, we will give the detail here.

Let $L$ be an open compact subgroup of $W$ and $w\in W$ be
fixed. Define a function $f_{w,L}: H_\beta(W)\rightarrow\C$ by 
\begin{align*}
f_{w, L}(w',z)=\begin{cases}\psi(z),
&\text{ if $w'\in A+w+L$},\\
0,&\text{ otherwise}.
\end{cases}
\end{align*}
Then $f_{w,L}\in S_A$. As in \cite[p.29]{MVW}, the space $S_A$ is
spanned by functions of this form.

Let $S'\subseteq S_A$ be a non-zero subspace invariant under the
action of $H_\beta(W)$. Let $w\in W$ be any. Then one can always find $f\in S'$ such
that $f(w,0)\neq 0$ by translating. By smoothness, there exists a open
compact subgroup
$L_w\subseteq W$, viewed as a subgroup of
$H_\beta(W)$, which fixes $f$. Let $L\subseteq L_w$ be an open subgroup
of $L_w$. Now by the theory of Fourier transforms, there exists
a locally constant function $\varphi$ on $A$ such that for all $w'\in
W$, we have
\[
\int_A\psi(\la w',a\ra)\varphi(a)\;da=
\begin{cases}
1,\quad \text{if $w'\in A+w+L$}\\
0,\quad \text{otherwise},
\end{cases}
\]
because  one can identify the Pontryagin dual of $W/A$ with $A$.
Now define a function $F_{\varphi,  f}:W\times F\rightarrow \C$ by
\begin{align*}
F_{\varphi, f}(w',z)
&:=\psi(z)\int_A f((w',0))(a,0))\varphi(a)\;da\\
&=\psi(z)\int_A f((a+w', \beta(w', a)))\varphi(a)\;da\\
&=\psi(z)\int_A f((a, \beta(w', a)-\beta(a,w'))(w', 0))\varphi(a)\;da\\
&=\psi(z)\left(\int_A \psi(\beta(w',a)-\beta(a,w'))\varphi(a)\;da\right) f(w',0)\\
&=\psi(z)\left(\int_A \psi(\la w',a\ra)\varphi(a)\;da\right) f(w',0)\\
&=\begin{cases}
\psi(z) f(w',0),& \text{if $w'\in A+w+L$}\\
0, &\text{otherwise}.
\end{cases}
\end{align*}
Hence $F_{\phi, f}$ is a scalar multiple of $f_{w, L}$. But $w$
is arbitrary and hence $L$ can be arbitrary. As we
mentioned, the functions of the form $f_{w, L}$ generate the space
$S_A$, \ie $S'=S$.
\end{proof}

\begin{Rmk}
The reader may wonder what would go wrong  if one uses the usual
Heisenberg group $H(W)$ but with the additive character $\psi_2$, so
that the character $(a, z)\mapsto \psi_2(z)$ is defined. (Also for
this matter, this is
essentially equivalent to defining the Heisenberg group using
$\la-,-\ra$ instead of $\frac{1}{2}\la-,-\ra$.) But if one
uses $\psi_2$ (or $\la-,-\ra$), the above proof would not work. Indeed, in this case
the induced representation would be reducible.
\end{Rmk}

We call $(\rho_W, S_A)$ the lattice model of the Heisenberg
representation. Accordingly, the Weil representation of
$\Spt(W)$ realized in the space $S_A$ is called the lattice model of
the Weil representation. (Strictly speaking the lattice model $S_A$ is
also dependent on the choice of $\beta$, but we avoid the notation
$S_{A,\beta}$, assuming $\beta$ is fixed.)

Let us note that each element $f\in S_A$ is a smooth compactly
supported (modulo the center) function $f:W\times F\rightarrow\C$ such that
\[
f(a+w,0)=\psi(-\beta(a,w))f(w,0)
\]
for $w\in W$ and $a\in A$. (Here note that
$(a+w,0)=(a,-\beta(a,w))(w,0)$.) 

For each $w\in W$, we let
\begin{equation}\label{E:s_w}
s_w=\text{ the unique function
in $S_A$ with } \supp(s_w)= (A+w)\times F
\text{ such that $s_w(w, 0)=1$}.
\end{equation}
This function plays an important role through the paper.

To make explicit the action of $M_g$ on the lattice model, we need to
use $\alpha^g$ as in Lemma \ref{L:alpha^g}.
Using $\alpha^g$, we can describe the action of $M_g$ on the space
$S_A$ of the lattice model explicitly as follows: For $f\in S_A$ and
$(w, z)\in H_\beta(W)$, define an endomorphism $M_g:S_A\rightarrow
S_A$ by
\begin{equation}\label{E:action_on_lattice_model}
M_g\circ f(w, z)=\int_{A/{A_g}}\psi(\beta(a, w))f(g^{-1}(a+w), z+\alpha^g(a+w))\;da
\end{equation}
where $w\in W, z\in F$, $g\in\Sp(W)$ and
\[
A_g=\{a\in A: ga\in A\text{ and } \alpha_g(a)\in\P^r\}.
\]
One can check that $A_g$ is a subgroup of $A$, because if $a, b\in
A_g$ and so $ga, gb\in A$, then $\alpha_g(a+b)=\beta(ga, gb)-\beta(a,
b)+\alpha_g(a)+\alpha_g(b)$ and each term here is in $\P^r$.
Note that $2(gA\cap A)\subseteq A_g$ and $A/A_g$ is a finite set. To
show the integral can be indeed defined over $A/A_g$, \ie the integral
is invariant under $A_g$, is a direct computation. 

\begin{Rmk}\label{R:important2}
Recall that each $g\in\Sp(W)$ has to be interpreted as $(g,
\alpha_g)\in\Ps(W)$, and so $g^{-1}$ is actually $(g,
\alpha_g)^{-1}=(g^{-1},\alpha^g)$ instead of
$(g^{-1},\alpha_{g^{-1}})$. (See Remark \ref{R:important}.) Hence
$M_{g^{-1}}$ is given by 
\[
M_{g^{-1}}\circ f(w, z)=\int_{A/{A_{g^{-1}}}}\psi(\beta(a, w))f(g(a+w), z+\alpha_g(a+w))\;da,
\]
where $A_{g^{-1}}=\{a\in A: g^{-1}a\in A\text{ and } \alpha^g(a)\in\P^r\}$.
\end{Rmk}

\begin{Prop}
The map $M_g$ is indeed well-defined, \ie $M_g\circ f\in S_A$ for all
$f\in S_A$.
\end{Prop}
\begin{proof}
Since the integral defining $M_g\circ f$ is a finite sum, it
is clear that it is smooth with compact support modulo the center. Hence we have only to
show $M_g\circ f(a'+w,0)=\psi(-\beta(a',w))M_g\circ f(w,0)$ for all $a'\in A$. But
\begin{align*}
M_g\circ f(a'+w, 0)&
=\int_{A/{A_g}}\psi(\beta(a, a'+w))f(g^{-1}(a+a'+w),
\alpha^g(a+a'+w))\;da\\
&=\int_{A/A_g}\psi(\beta(a-a', a'+w))f(g^{-1}(a+w),
\alpha^g(a+w))\;da\\
&=\int_{A/A_g}\psi(\beta(a,
w)+\beta(a,a')-\beta(a',a')-\beta(a',w))f(g^{-1}(a+w),\alpha^g(a+w))\;da\\
&=\psi(-\beta(a',w))\int_{A/{A_g}}\psi(\beta(a, w))f(g^{-1}(a+w),
z+\alpha^g(a+w))\;da\\
&=\psi(-\beta(a',w))M_g\circ f(w,0),
\end{align*}
where for the second equality we used the invariance of the measure on
$A$ and for the fourth, we used $\beta(a,a')\in\P^r$ and
$\beta(a',a')\in\P^r$.
\end{proof}

\begin{Prop}
For each $g\in\Sp(W)$, the map $M_g$ is not identically zero.
\end{Prop}
\begin{proof}
Consider the function $s_0\in S_A$, \ie set $w=0$ in
$(\ref{E:s_w})$. Then $s_0(a,0)=1$ for all $a\in A$. We have
\begin{align*}
M_g\circ s_0(0,0)&=
\int_{A/A_g}s_0(g^{-1}a,\alpha^g(a))da\\
&=\int_{gA\cap A/A_g}s_0(g^{-1}a,\alpha^g(a))da\\
&=\int_{gA\cap A/A_g}\psi (\alpha^g(a))da.
\end{align*}
Now the map
$\psi\circ \alpha^g$ is a non-degenerate character of second
degree on the finite group $gA\cap A/A_g$ in the sense of
\cite{Weil}. Hence by \cite[Theorem
A.2 (5)]{Rao}, we have
\[
\int_{gA\cap A/A_g}\psi (\alpha^g(a))da=|gA\cap
A/A_g|^{1/2}\gamma(\psi\circ\alpha^g),
\]
where $\gamma(\psi\circ\alpha^g)$ is the Weil index of
$\psi\circ\alpha^g$. (See \cite[Appendix]{Rao} for details.) In
particular, it is non-zero. Thus $M_g\circ s_0(0,0)\neq 0$. The
proposition follows.
\end{proof}

Then we have
\begin{Prop}
For the above defined $M_g$, we have $M_g\circ\rho_W(h)=\rho_W(g\cdot
h)\circ M_g$ for all $h\in H_\beta(W)$,  namely the element $(g, M_g)$ is
indeed in $\Spt(W)$ and $M_g$ defines the action for the Weil representation.
\end{Prop}
\begin{proof}
Note that since $\rho_W$ is irreducible, if we can show $M_g\circ\rho_W(h)=\rho_W(g\cdot
h)\circ M_g$ for all $h$, it will imply $M_g$ is invertible and hence $(g, M_g)$ is
indeed in $\Spt(W)$.

Let $(w, z)\in H_\beta(W)$, and $f\in S_A$. Also let $h=(w', z')\in
H_\beta(W)$. Recalling $g\cdot h=g\cdot (w',z')=(gw',
z'+\alpha_g(w'))$, we have
\begin{align*}
&(\rho_W(g\cdot h)\circ M_g)\circ f (w, z)\\
&=\rho_W(g\cdot h)(M_g\circ f)(w,z)\\
&=M_g\circ f((w,z)(gw', z'+\alpha_g(w')))\\
&=M_g\circ f(w+gw', z+z'+\alpha_g(w')+\beta(w, gw'))\\
&=\int\psi(\beta(a, w+gw'))f(g^{-1}(a+w+gw'), z+z'+\alpha_g(w')+\beta(w,
gw')+\alpha^g(a+w+gw'))\;da\\
&=\int\psi(\beta(a, w))f(g^{-1}(a+w)+w', z+z'+\alpha_g(w')+\beta(w,
gw')+\alpha^g(a+w+gw')+\beta(a, gw'))\;da\\
&=\int\psi(\beta(a, w))f(g^{-1}(a+w)+w',
z+z'+\alpha_g(w')+\alpha^g(a+w+gw')+\beta(w+a, gw'))\;da,
\end{align*}
where all the integrals are over $A/A_g$.

On the other hand, 
\begin{align*}
&(M_g\circ \rho_W(h))f(w, z)\\
&=\int\psi(\beta(a, w))\rho(w', z')f(g^{-1}(a+w),
z+\alpha^g(a+w))\;da\\
&=\int\psi(\beta(a, w)) f((g^{-1}(a+w), z+\alpha^g(a+w))(w',
z'))\;da\\
&=\int\psi(\beta(a, w))f(g^{-1}(a+w)+w',
z+z'+\alpha^g(a+w)+\beta(g^{-1}(a+w), w'))\;da,
\end{align*}
where again all the integrals are over $A/A_g$.

In order for us to show that those two are equal, it suffices to show
\[
\alpha_g(w')+\alpha^g(a+w+gw')+\beta(w+a, gw')=\alpha^g(a+w)+\beta(g^{-1}(a+w), w'),
\]
namely
\[
\alpha^g(a+w+gw')-\alpha^g(a+w)+\alpha_g(w')=\beta(g^{-1}(a+w), w')-\beta(w+a, gw').
\]
But this follows because $\alpha_g(w')=-\alpha^g(gw')$ and
$\alpha^g\in\Sigma_{g^{-1}}$.
\end{proof}

Let us define
\[
\Gamma_A:=\{g\in\Sp(W): gA\subseteq A\}.
\]
One can see that $\Gamma_A$ is an open compact subgroup of
$\Sp(W)$. Note that the condition $gA\subseteq A$ implies $gA=A$
because $g$ is an isometry and hence preserves volume. (This also
applies to any lattice. See for
example \cite[\S 82:12]{O'Meara}.) Also we have
$\Gamma_A=\Sp_{2n}(\O)$.
If $g\in\Gamma_A$, then 
\begin{align*}
M_g\circ f(w, z)&=\int_{A/A_g}\psi(\beta(a, w)-\beta(g^{-1}a,
g^{-1}w)+\alpha^g(a+w))f(g^{-1}w, z)\;da\\
&=\int_{A/A_g}\psi(\alpha^g(a)+\alpha^g(w))f(g^{-1}w, z)\;da\\
&=\int_{A/A_g}\psi(\alpha^g(a))f(g^{-1}w, z+\alpha^g(w))\;da\\
&=\left(\int_{A/A_g}\psi(\alpha^g(a))\;da\right)f(g^{-1}w, z+\alpha^g(w)).
\end{align*}
where to obtain the second equality we used the property that
$\alpha^g\in\Sigma_{g^{-1}}$. This implies
\begin{Prop}\label{P:Gamma_A_action}
Each $g\in\Gamma_A$ acts on $f\in S_A$ by (non-zero scalar multiple
of) translation. To be more precise,
\[
M_g\circ f(w,z)=\left(\int_{A/A_g}\psi(\alpha^g(a))\;da\right)\cdot 
f(g^{-1}w, z+\alpha^g(w)).
\]
\end{Prop}

From the above integral formula for $M_g$, it is important to know when we
have $\psi(\alpha^g(a))=1$ for all $a\in A$ or equivalently
$\alpha^g(a)\in\P^r$ for all $a\in A$. For
this purpose, let us start with
\begin{Lem}\label{L:open_closed}
Let $B$ be a (not necessarily self-dual) lattice of $W$. Then the set
\[
G_B:=\{g\in\Sp(W): \alpha_g(w)\in\P^r\text{ for all } w\in B\}
\]
is open and closed. (We do not know if it is a group.)
\end{Lem}
\begin{proof}
For each $w\in B$, consider the continuous map $f_w:\Sp(W)\rightarrow F$ defined
by $f_w(g)=\alpha_g(w)$. Then
\[
G_B=\bigcap_{w\in B}f_w^{-1}(\P^r).
\]
Note that $f^{-1}_w(\P^r)$ is open and closed, so $G_B$ is
closed. Next by looking at the description of $\alpha_g(w)$ in (\ref{E:alpha_g}), one can
see that $f^{-1}_w(\P^r)\subseteq f_{w'}^{-1}(\P^r)$ for all $w'\in\O
w$. Since $B$
is compact, we have $B=\bigcup_i\O w_i$ for some finite union. So
$G_B=\bigcap_i f_{w_i}^{-1}(\P^r)$, which is open.
\end{proof}

Now let us define
\[
\Gamma_A^\circ:=G_A\cap\Gamma_A=
\{g\in\Gamma_A:\alpha_g(a)\in\P^r\text{ for all $a\in A$}\}.
\]
Here let us emphasize that each
$g\in\Gamma_A\subseteq\Sp(W)$ should be interpreted as
$(g,\alpha_g)\in\Ps(W)$. (See Remark \ref{R:important}.)

\begin{Lem}\label{L:subgroup}
$\Gamma_A^\circ$ is an open compact subgroup of $\Sp(W)$.
\end{Lem}
\begin{proof}
By the above lemma and the fact that $\Gamma_A$ is open and compact,
one can conclude that $\Gamma_A^\circ$ is open and compact.

To show it is a subgroup, let $g, h\in\Gamma_A^\circ$. For $a\in A$, we have
$\alpha_{gh}(a)=h^{-1}\cdot\alpha_g(a)+\alpha_h(a)
=\alpha_g(ha)+\alpha_h(a)\in\P^r$
because $ha\in A$. Hence $gh\in\Gamma_A^\circ$. Now let
$g\in\Gamma_A^\circ$. To show $g^{-1}\in\Gamma_A^\circ$, one needs to show
$\alpha^g(a)\in\P^r$ for all $a\in A$. (Here what is needed is not
$\alpha_{g^{-1}}(a)\in\P^r$! See Remark \ref{R:important}.) But since
$\alpha^g(a)=-\alpha_g(g^{-1}a)$, and $g^{-1}a\in A$, we have
$-\alpha_g(g^{-1}a)\in\P^r$. Hence $\Gamma_A^\circ$ is a subgroup.
\end{proof}

Now for the integral defining $M_g$, if we choose the measure $da$ so
that the volume of $A/A_g$ is $1$, one has
\[
M_g\circ f(w, z)=f(g^{-1}w, z+\alpha^g(w))
\quad\text{for all $g\in\Gamma_A^\circ$}.
\]
Note that
\[
M_{g^{-1}}\circ f(w, z)=f(gw, z+\alpha_g(w)).
\]
by Remark \ref{R:important2}.
One can check that for $g, h\in\Gamma_A^\circ$, we have
$M_{gh}\circ f =M_h\circ (M_g\circ f)$ by using
$\alpha_{gh}=h^{-1}\cdot\alpha_g+\alpha_h$. This shows
$\Gamma_A^\circ$ splits in the metaplectic cover $\Spt(W)$. 

Indeed if the residue characteristic of $F$ is odd, one can see
$\psi(\alpha^g(a))=1$ for any $g\in\Gamma_A$ and $a\in A$. This is
because from the explicit description of $\alpha_g$ as in
(\ref{E:alpha_g}) together with the fact that $\frac{1}{2}$ is a unit
in $\O$, one can see that all the three terms in the definition of
$\alpha_g(a)$ are in $\P^r$. Hence $\Gamma_A^\circ=\Gamma_A$. So this
explains the well-known splitting of $\Sp(\O)$ in $\Spt(W)$.

If the residue characteristic of $F$ is even, we no
longer have  $\psi(\alpha^g(a))=1$ for every $g\in\Gamma_A$ and
$a\in A$. Yet, the above lemma shows that for
a sufficiently small open compact subgroup $\Gamma_A^\circ$of
$\Gamma_A$, we do have $\psi(\alpha^g(a))=1$ for any
$g\in\Gamma_A^\circ$ and any
$a\in A$. Hence we have the analogous splitting of this open compact subgroup,
which also explains the well-known fact that a certain open compact
subgroup of $\Sp(W)$ splits in $\Spt(W)$ for the case of even residual
characteristic. 

\begin{Rmk}
Not only the group $\Gamma_A^\circ$ but also various other subgroups
of $\Sp(W)$ are known to be split in the metaplectic cover
$\Spt(W)$. Whenever $H$ is a subgroup of $\Sp(W)$ which splits in
$\Spt(W)$, for each $h\in H$ and $f\in S_A$ we usually denote
$M_h\circ f$ by $\omega_\psi(h)f$ or simply $\omega(h)f$ because the
additive character $\psi$ is fixed throughout the paper.
\end{Rmk}


\section{\bf On lattices}\label{S:lattices}


In this section we let $(\V, \la-,-\ra)$ be an $\epsilon$-Hermitian
space over $E$, where $E$ is either $F$ or a quadratic extension of
$F$. In particular, we have
\[
\la cv, c'v'\ra= c\bar{c'}\la v, v'\ra\quad\text{ and }\quad
\overline{\la v, v'\ra}=\epsilon\la v', v\ra
\] 
for $v, v'\in \V$ and $c, c'\in E$. (Recall from the notation section that for each $c\in E$, we
denote $\bar{c}=c$ if $E=F$, and $\bar{c}=\tau(c)$ where $\tau$
is the nontrivial element in $\Gal(E/F)$ if $E$ is a quadratic
extension of $F$.) We always assume that $E$ is unramified over $F$
when $E\neq F$.

By a lattice $L$ of $\V$, we mean a free $\OE$-module
of rank equal to $\dim \V$. For any lattice $L\subseteq\V$, we define the dual
lattice $L^\perp$ with respect to an integer $r$ by
\[
L^\perp=\{v\in\V:\la v, l\ra\in\PE^r \text{ for all $l\in L$} \}.
\]
(In this paper $r$ is usually reserved for the exponential conductor
of $\psi$ but in this section we use $r$ for any fixed integer.) A
lattice $L$ is called a self-dual lattice (with respect to
$r$) if $L^\perp=L$. 

Not every $\epsilon$-Hermitian space has a
self-dual lattice, and even when it does, we
sometimes need some restriction on $r$. To be specific, we have

\begin{Lem}\label{L:self-dual_lattice}
An $\epsilon$-Hermitian space $(\V, \la\;,\;\ra)$ admits a self-dual
lattice (with an occasional restriction on $r$) if it is one of the following:
\begin{enumerate}[(a)]
\item $\V$ is symplectic, namely $E=F$ and $\epsilon=-1$. ($r$ can be any.)
\item $\V$ is symmetric, namely $E=F$ and $\epsilon=1$,  where the
  anisotropic part $\V^a$ is one of the following:
\begin{enumerate}[-]
\item $\V^a=0$; ($r$ can be any.)
\item $\V^a=F$ and $\la a, b\ra=\eta ab$ for $a, b\in F$ where
  $\eta\in\O^\times$; ($r$ has to be even.)
\item $\V^a=F'$ where $F'$ is an unramified quadratic extension of
  $F$ equipped with the norm form, namely for $x, y\in F'$, we have
  $\la x, y\ra=\frac{1}{2}(x\bar{y}+\bar{x}y)$ where the bar is the
  conjugation for the quadratic extension $F'/F$. ($r$ has to be even.)
\end{enumerate}
\item $\V$ is Hermitian, namely $E$ is a quadratic unramified
  extension over $F$ and
  $\epsilon=\pm1$, where the anisotropic part $\V^a$ is one of the
  following:
\begin{enumerate}[-]
\item $\V^a=0$; ($r$ can be any.)
\item $\V^a=E$ equipped with the norm form if $\epsilon=1$, namely
  $\la x, y\ra=x\bar{y}$ for $x, y\in E$, and
  $\eta$ times the norm form if $\epsilon=-1$ where $\eta$ is an
  element in $\OE^\times$ such that $\bar{\eta}=-\eta$, namely $\la x,
  y\ra=\eta x\bar{y}$. ($r$ has to be even.)
\end{enumerate}
\end{enumerate}
\end{Lem}
\begin{proof}
This list is as in \cite[p.100]{MVW}, although there it is always assumed
$r=0$. But for later convenience, let us describe the self-dual
lattices for all the cases in detail.

\begin{enumerate}[(a)]
\item $\V$ is symplectic; Then $\dim \V=\text{even}=2n$ and $\V$ admits a basis
  $\{e_1,\dots, e_n, f_1,\dots,f_n\}$ so that
  $\la e_i, f_j\ra=\delta_{ij}\varpi^r$, $\la e_i, e_j\ra=\la
  f_i, f_j\ra=0$, and further 
\[
L=\Span_{\O}\{e_1,\dots, e_n, f_1,\dots,f_n\}.
\]
\item $\V$ is symmetric;
\begin{enumerate}[-]
\item $\V^a=0$; Then $\dim \V=\text{even}=2n$ and $\V$
  admits a basis $\{e_1,\dots, e_n, f_1,\dots,f_n\}$ so that
  $\la e_i, f_j\ra=\delta_{ij}\varpi^r$ and $\la e_i, e_j\ra=\la
  f_i, f_j\ra=0$, and further 
\[
L=\Span_{\O}\{e_1,\dots, e_n, f_1,\dots,f_n\}.
\]
\item $\V^a=F$; Then $\dim \V=\text{odd}=2n+1$ and $\V$
  admits a basis $\{e_1,\dots, e_n, f_1,\dots,f_n, v\}$ with
  $\V^a=\Span_F\{v\}$ where $v=\varpi^{r/2}$ so that
  $\la e_i, f_j\ra=\delta_{ij}\varpi^r$, $\la e_i, e_j\ra=\la
  f_i, f_j\ra=\la e_i, v\ra=\la f_i, v\ra=0$, and $\la v,
  v\ra=\eta\varpi^{r}$ (assuming $r$ is  even), and further 
\[
L=\Span_{\O}\{e_1,\dots, e_n, f_1,\dots,f_n, v\}.
\]
\item $\V^a=F'$; First note that since $F'$ is unramified over $F$,
  we can write $F'=F\oplus\eta F$ where $\eta\in F'$ is such that $\eta\bar{\eta}\in
  \O^\times$. Then $\dim \V=\text{even}=2n+2$ and $\V$ admits a basis
  $\{e_1,\dots, e_n, f_1,\dots,f_n, v_1, v_2\}$ with
  $\V^a=\Span_F\{v_1, v_2\}$ where $v_1=\varpi^{r/2}$ and
  $v_2=\eta\varpi^{r/2}$ so that
 $\la e_i, f_j\ra=\delta_{ij}\varpi^r$, $\la e_i, e_j\ra=\la
  f_i, f_j\ra=\la e_i, v\ra=\la f_i, v_i\ra=\la v_1,
  v_2\ra=0$, $\la v_1,v_1\ra=\varpi^r$ and $\la
  v_2,v_2\ra=-\eta\bar{\eta}\varpi^r$, and 
 further 
\[
L=\Span_{\O}\{e_1,\dots, e_n, f_1,\dots,f_n, v_1, v_2\}.
\]
\end{enumerate}
\item $\V$ is Hermitian;
\begin{enumerate}[-]
\item $\V^a=0$; Then $\dim_E\V=\text{even}=2n$, and $\V$ admits an $E$-basis
  $\{e_1,\dots, e_n, f_1,\dots,f_n\}$ so that
  $\la e_i, f_j\ra=\delta_{ij}\varpi^r$ and $\la e_i, e_j\ra=\la
  f_i, f_j\ra=0$, and further 
\[
L=\Span_{\OE}\{e_1,\dots, e_n, f_1,\dots,f_n\}.
\]
\item $\V^a=E$;  We let
  $v=\varpi^{r/2}$. We have $\la v,
  v\ra=u\varpi^r$ for some unit
  $u\in\OE^\times$. Then $\dim_E \V=\text{odd}=2n+1$,
and
$\V$ admits an $E$-basis 
$\{e_1,\dots, e_n, f_1,\dots,f_n, v\}$ so that
$\la e_i, f_j\ra=\delta_{ij}\varpi^r$, $\la e_i, e_j\ra=\la
  f_i, f_j\ra=\la e_i, v\ra=\la f_i, v\ra=0$ and $\la v,
  v\ra=u\varpi^r$, and further 
\[
L=\Span_{\OE}\{e_1,\dots, e_n, f_1,\dots,f_n, v\}.
\]
\end{enumerate}
\end{enumerate}
\end{proof}

From this lemma, one can conclude
\begin{Lem}\label{L:lattice_decomposition}
Let $L\subseteq\V$ be a self-dual lattice of $\V$. Then there is a
decomposition $\V=\V^+\oplus\V^a\oplus\V^-$, where
$\V^a$ is the anisotropic part and both $\V^+$ and $\V^-$
are totally isotropic of the same dimension, so that $\V^+\oplus
\V^-$ is a product of copies of the $\epsilon$-Hermitian hyperbolic
planes, such that
\[
L=\V^+\cap L\oplus\V^a\cap L\oplus\V^-\cap L.
\]
We often write
\[
L^+=\V^+\cap L,\quad L^a=\V^a\cap L,\quad L^-=\V^-\cap L.
\]
\end{Lem}

\quad\\

We will recall some of the properties
of self-dual lattices, all of which are essentially in \cite[p.107-112]{MVW},
though in \cite{MVW} it is always assumed $r=0$. 

Note that the quotient $L/\varpi L$ is viewed as an
$\epsilon$-Hermitian space over the residue field $\OE/\varpi\OE$ by
reducing $\varpi^{-r}\la-,-\ra$ mod $\varpi\OE$. For each
$v\in\V$, we denote by $\bar{v}$ the image of $v$ in $L/\varpi L$.

Let us mention a couple of properties on the $\epsilon-$Hermitian spaces over
the finite field $\OE/\varpi\OE$.

\begin{Lem}\label{L:Finite_Witt}
Let $\overline{L}:=L/\varpi L$ be the $\epsilon-$Hermitian spaces over
the finite field $\OE/\varpi\OE$.
\begin{enumerate}
\item Assume $\overline{X}$ is a totally isotropic subspace of
  $\overline{L}$. Then there exist subspaces $\overline{Y}$ and
  $\overline{L}^\circ$ sucht that $\overline{Y}$ is totally isotropic
  with $\dim\overline{Y}=\dim\overline{Y}$, and we have the
  decomposition
  $\overline{L}=\overline{X}\oplus\overline{L}^\circ\oplus\overline{Y}$
  where $\overline{X}\oplus\overline{Y}$ is orthogonal to
  $\overline{L}^\circ$.
\item Let $\overline{X}$ and $\overline{X}'$ be two subspaces of
  $\overline{L}$ with $\dim\overline{X}=\dim\overline{X}'$. Then any
  isometry $\overline{X}\rightarrow\overline{X}'$ can be extended to
  an isometry on $\overline{L}$.
\end{enumerate}
\end{Lem}
\begin{proof}
Both of them are well-known when the characteristic of the residue
field $\OE/\varpi\OE$ is odd. When the characteristic is even, it does
not seem to be well-known. First of all, let us denote the corresponding form
by $b(-,-):=\varpi^{-r}\la-,-\ra$. If $E=F$, then the form $b$ is
always symmetric because $\epsilon=1$ in $\OE/\varpi\OE$. Clearly $b$
is non-degenerate. Then part (1) is 
\cite[Proposition 1.2.2]{Kitaoka} and part (2) is \cite[Corollary
1.2.1]{Kitaoka}.

If $E\neq F$, then Dieudonn\'e in \cite[p. 21]{Dieudonne} has shown
that these two properties hold if $b(x,x)$ is a trace in
$\OE/\varpi\OE$ for all $x\in\overline{L}$, namely for each
$x\in\overline{L}$ there exists $a\in\OE/\varpi\OE$ such that
$b(x,x)=a+\bar{a}$, where $\bar{a}$ is the Galois conjugate of $a$ for
the quadratic extension $\OE/\varpi\OE$ over $\O/\varpi\O$. But
considering $\epsilon=1$, we always have $\overline{b(x,x)}=b(x,x)$ and
so $b(x,x)\in\O/\varpi\O$. The condition is satisfied because the
trace map $\OE/\varpi\OE\rightarrow\O/\varpi\O$ is surjective.
\end{proof}

\begin{Lem}\label{L:basis_extension}
Let $v_1,\dots,v_k\in L$ be such that the reductions
$\bar{v}_1,\dots,\bar{v}_k$ in $L/\varpi L$ are linearly independent
over the residue field. Then $\{v_1,\dots,v_k\}$ can be extended to a
basis of $L$.
\end{Lem}
\begin{proof}
This is an elementary exercise. 
\end{proof}

The following lemma, which is an extension of \cite[Proposition II.2,
p.107]{MVW} is crucial to our computations.

\begin{Lem}\label{L:lattice1}
For a self-dual lattice $L$ (with respect to $r$) of $\V$, let $v_1,\dots, v_s\in L$,
$t_1,\dots, t_s\in\Z$ and $M=(m_{ij})$ an $s\times s$  matrix with
coefficients in $\OE$. Suppose
\begin{enumerate}
\item $\bar{v}_1,\dots, \bar{v}_s$ are linearly independent over the
  residue field;
\item $1+e\leq t_1\leq\cdots\leq t_s$;
\item $m_{ij}=\epsilon\overline{m_{ji}}$ for all $1\leq i,j\leq s$;
\item $m_{ij}\equiv\la v_i,v_j\ra\mod \PE^{r+t_i}$ for all $1\leq
  i\neq j\leq s$;
\item $m_{ii}\equiv\la v_i,v_i\ra\mod \PE^{r+t_i+e}$ for all $i\in\{1,\dots,s\}$.
\end{enumerate}
Then there exist $v'_1,\dots, v'_s\in L$ such that
\begin{enumerate}
\item[(6)] $v_i'-v_i\in\varpi^{t_i}L$ for all $i$;
\item[(7)] $m_{ij}=\la v_i', v_j'\ra$ for all $1\leq i,j\leq s$.
\end{enumerate}
\end{Lem}
\begin{proof}
When the residue characteristic of $F$ is odd, this is
\cite[Proposition II.2, p.107]{MVW} except that in \cite{MVW} it
is always assumed $r=0$. For the case of even residual characteristic,
one needs to slightly modify the argument there, and we need the
condition (5), which can be absorbed by the condition (4) in the case
of odd residual characteristic. In any case, since the proof is only a
slight modification of the one given in \cite{MVW}, we will repeat
only the
essentially point. The basic idea is to construct a sequence of vectors $v_i(t)$
for each $i\in\{1,\dots,s\}$ where $t\in\Z_{\geq1}$ such that
\begin{enumerate}[(a)]
\item $v_i(t)-v_i\in\varpi^{t_i}L$;
\item $v_i(t)-v_i(t-1)\in\varpi^{t+r-1}L$ for $t\geq 2$;
\item $\la v_i(t), v_j(t)\ra\equiv m_{ij}\mod\PE^{t+r}$ for all $1\leq
  i\neq j\leq s$.
\item $\la v_i(t), v_i(t)\ra\equiv m_{ii}\mod\PE^{t+r+e}$ for all
  $\in\{1,\dots,s\}$.
\end{enumerate}

The condition (b) guarantees the sequence $v_i(t)$ is Cauchy and hence
converges to some $v_i'$, and the conditions (a) (c) and (d) guarantee
$v_i'$ has the desired property. 

The sequence $v_i(t)$ is constructed by recursion as follows. First by
Lemma \ref{L:basis_extension}, we can
extend $\{v_1,\dots,v_s\}$ to a basis $\{v_1,\dots,v_n\}$ of
$L$. Choose a dual basis $\{v_1^\ast,\dots,v_n^\ast\}$ of
  $L$, so that $\la v_i, v_j^\ast\ra=\varpi^r\delta_{ij}$. Define $v_i(t)$ by 
\begin{align*}
v_i(1)&=v_i\\
v_i(t)&=\begin{cases}
v_i(t-1)&\text{ if $t\leq t_i$}\\
v_i(t-1)+\sum_{k=i}^r\varpi^{t+r-1}a_{ki}(t)v_i^\ast&\text{ if $t> t_i$},
\end{cases}
\end{align*}
where
\begin{align*}
a_{ki}(t)&=\varpi^{1-t-r}(m_{ki}-\la v_k(t-1), v_i(t-1)\ra)\text{ if
  $i<k$};\\
a_{ii}(t)&=\frac{1}{2}\varpi^{1-t-r}(m_{ii}-\la v_i(t-1), v_i(t-1)\ra).
\end{align*}
(Let us mention that in \cite[p.108]{MVW} there is a typo in the
definition of $v_i(t)$. The summation has to start with $i$ instead of $i+1$.)
Note that $a_{ji}(t)\in\OE$ thanks to the condition (c) and (d). Also
note that to obtain (c) and (d) at each step, one needs $t_i\geq
1+e$. By definition of $v_i(t)$, one
has $v_i(t)-v_i\in\varpi^{t_1}L$ if $t\leq t_i$, and
$v_i(t)-v_i=\varpi^{t-1}L\subset\varpi^{t_i}L$ if
$t>t_i$. So the condition (a) is satisfied. 
\end{proof}

This lemma is very unfortunate in that the restriction $t_i\geq 1+e$
will not allow us to apply many of the computations in \cite{MVW} to
the case of even residual characteristic. However, if we assume $\V$
is symplectic, we have

\begin{Lem}\label{L:lattice_symplectic}
If $\V$ is symplectic in the above lemma, one can assume $t_i\geq 1$
instead of $t_i\geq 1+e$, and can suppress the condition (5).
\end{Lem}
\begin{proof}
If $\V$ is symplectic, one can simply take $a_{ii}(t)=0$ for all
$t$. (Of course, always $m_{ii}=0$.)
\end{proof}

This implies

\begin{Lem}\label{L:lattice_lin_ind}
For a self-dual lattice $L$ (with respect to $r$) of $\V$, let
$v_1,\dots, v_n\in L$, and $t\in\Z^{>0}$. Further assume $\V$ is
symplectic if the residue characteristic of $F$ is even. Suppose
\begin{enumerate}
\item $\bar{v}_1,\dots,\bar{v}_n$ are linearly independent over the
  residue field;
\item For all $i, j\in\{1,\dots,s\}$, $\la v_i,v_i\ra\equiv 0\mod \PE^{r+t}$.
\end{enumerate}
Then there exist elements $v'_1\dots v_n'\in L$ and subspaces $X,
\V^\circ, Y$ of $\V$ such that
\begin{enumerate}
\item[(3)] $\{v'_1\dots v'_n\}$ is a basis of $X$ over $E$;
\item[(4)] we have the orthogonal decomposition
  $W=X\oplus\V^\circ\oplus Y$ such that $X\oplus Y$ is totally
  isotropic;
\item[(5)] $L=L\cap X\oplus L\cap\V^\circ\oplus L\cap Y$;
\item[(6)] for all $i=1,\dots, n$, we have $v'_i-v\in \varpi^t L$.
\end{enumerate}
\end{Lem}
\begin{proof}
This can be proven in the same way as \cite[Corollary II.3]{MVW} using
the previous lemma and Lemma \ref{L:Finite_Witt}.
\end{proof}

One reason we have to assume $\V_2$ is symplectic in our main theorem
is the unavailability of this lemma for the other types of spaces. 

Lemma \ref{L:lattice_symplectic} also implies

\begin{Lem}\label{L:lattice_Witt}
Let $L$ be a self-dual lattice of $\V$, where $\V$ is either
symplectic or symmetric with $\V^a=0$ when the residue characteristic
is even. Let $v_1,\dots,v_s$ and $v_1',\dots, v_s'$ are vectors in
$\V$ such that
\begin{enumerate}
\item $\bar{v}_1,\dots,\bar{v}_s'$ are linearly independent over the
  residue field;
\item $\bar{v'}_1,\dots,\bar{v'}_s'$ are linearly independent over the
  residue field;
\item $\la v_i,v_j\ra=\la v_i',v_j'\ra$ for all $i,j\in\{1,\dots,s\}$.
\end{enumerate} 
Then there exists $u\in U(\V)$ such that $u(L)=L$ and $u(v_i)=v_i'$
for all $i\in\{1,\dots,s\}$.
\end{Lem}
\begin{proof}
This is \cite[Corollary II.5, p.111]{MVW} when the residue
characteristic is odd. If the residue characteristic is even and $\V$
is symplectic, one can prove it in the same way as the case of
odd residual characteristic by using Lemma \ref{L:lattice_symplectic}. For the
other case, unfortunately, the same proof does not work because of the
restriction $t_i\geq 1+e$ in Lemma \ref{L:lattice1}. However if $\V$
is symmetric and $\V^a=0$, then this lemma is simply Witt's extension
theorem for self-dual lattices, which is known to hold. (See, for example, Corollary
5.4.1 of \cite{Kitaoka}.)
\end{proof}

As the last thing in this section, let us introduce the notion of
{\it admissible lattices} and some of their properties.

\begin{Def}
Assume $\V$ is such that $\V^a=0$. We say a (not necessarily self-dual)
lattice $L\subseteq\V$ is ``admissible'' if $L=L^++L^-$, where recall
$L^+=L\cap\V^+$ and $L^-=L\cap\V^-$.
\end{Def}

Note that for any lattice $L$ with $L^a=0$, we always have $L\subseteq L^++L^-$, but the
inclusion might be strict. A self-dual lattice is always
admissible.

We need to quote a few lemmas:

\begin{Lem}\label{L:dual_is_admissible}
Assume $\V$ is such that $\V^a=0$. A lattice $L$ of $\V$ is admissible
if and only if $L^\perp$ is.
\end{Lem}
\begin{proof}
Assume $L$ is admissible, so $L=L^++L^-$. Let $l\in L^\perp$. We need
to show $l^+\in L^\perp$. Let $m=m^++m^-\in L$. Note that since $L$ is
admissible, $m^-\in L$, and so $\la l, m^-\ra\in\P^r$.  So we have
\[
\la l^+, m\ra=\la l^+, m^++m^-\ra=\la l^+, m^-\ra
=\la l^++l^-, m^-\ra=\la l, m^-\ra\in\P^r.
\]
Thus $l^+\in L^\perp$. So $(L^\perp)^+\subseteq L^\perp$. Similarly
$(L^\perp)^-\subseteq L^\perp$. Hence $(L^\perp)^++(L^\perp)^-\subseteq
L^\perp$, which shows $L^\perp$ is admissible.

Conversely assume $L^\perp$ is admissible. From the above argument,
$L^{\perp\perp}$ is admissible. But $L^{\perp\perp}=L$, so $L$ is admissible.
\end{proof}

\begin{Lem}
Let $L_1$ be a self-dual lattice of $\V$ and $L\subseteq L_1$ a
sublattice of $L_1$. Assume $\dim\V=n$. Then there exist a basis
$\{e_1,\dots, e_n\}$ of $L_1$ and an integer $s$ with $0\leq s\leq n$
such that
\[
\{e_1,\dots,e_s,\varpi^{t_{s+1}}e_{s+1},\dots,\varpi^{t_n}e_n\}
\]
is a basis of $L$, where $t_i\geq1$ for $s+1\leq i\leq n$.
\end{Lem}
\begin{proof}
This is a part of Lemma in \cite[p.112]{MVW}. 
\end{proof}

This lemma implies
\begin{Lem}
Assume $\V$ is such that $\V^a=0$. Let $L_1$ be a self-dual lattice of
$\V$ and $L\subseteq L_1$ an admissible
sublattice of $L_1$. Assume $\dim\V=2n$. Then there exist a basis
$\{e_1,\dots, e_n\}$ of $L_1^+$, a basis $\{f_1,\dots, f_n\}$ of
$L_1^-$ and integers $s^+, s^-$ with $0\leq s^+\leq n$ and $0\leq s^-\leq n$
such that
\[
\{e_1,\dots,e_{s^+},\varpi^{t_{{s^+}+1}}e_{{s^+}+1},\dots,\varpi^{t_n}e_n\}
\]
is a basis of $L^+$, where $t_i\geq1$ for $s^++1\leq i\leq n$, and 
\[
\{f_1,\dots,f_{s^-},\varpi^{u_{{s^-}+1}}f_{{s^-}+1},\dots,\varpi^{u_n}f_n\}
\]
is a basis of $L^-$, where $u_i\geq1$ for $s^-+1\leq i\leq n$.
\end{Lem}
\begin{proof}
By the above lemma, there exist a basis $\{v_1,\dots,v_{2n}\}$ of
$L_1$ and an integer $s$ such that
$\{v_1,\dots,v_s,\varpi^{t_{s+1}}v_{s+1},\dots,\varpi^{t_n}v_{2n}\}$
is a basis of $L$. Since $L$ is admissible, each
$\varpi^{t_i}v_i^+$ (where $t_i=0$ if $i\leq s$) is in $L^+$, and
hence
$\{v_1^+,\dots,v_s^+,\varpi^{t_{s+1}}v_{s+1}^+,\dots,\varpi^{t_n}v_{2n}^+\}$
is a generator of the free $\OE$-module $L^+$. Then we can shrink it to
a basis of $L^+$. Similarly for $L^-$.
\end{proof}


\section{\bf Unramified dual pairs}\label{S:dual_pair}


For $i=1, 2$, let $(\V_i, \la\;,\;\ra_i)$ be an
$\epsilon_i$-Hermitian space over $E$ where $\epsilon_i\in\{\pm 1\}$.
If $\epsilon_1\epsilon_2=-1$, the space
\[
W=\V_1\otimes_F \V_2
\]
becomes a symplectic space of dimension $\dim_F \V_1\cdot\dim_F \V_2$
with the symplectic form defined by
\[
\la v_1\otimes v_2, v_1'\otimes v_2'\ra=\tr_{E/F}({\la v_1,
v_1'\ra}_1\overline{\la v_2, v_2'\ra}_2).
\]
We have the natural map $U(\V_1)\times U(\V_2)\rightarrow\Sp(W)$,
and say the pair $(U(\V_1), U(\V_2))$ is a dual pair.

Assume both $(\V_1,\la\;,\;\ra_1)$ and
$(\V_2,\la\;,\;\ra_2)$ admit self-dual lattices $L_1\subseteq \V_1$ and
$L_2\subseteq\V_2$ with respect to the integers $r_1$ and $r_2$,
respectively. We fix a decomposition
\[
\V_i=\V_i^+\oplus \V_i^a\oplus \V_i^-
\]
as in Lemma \ref{L:lattice_decomposition}. Notice that the lattice 
\[
A:=L_1\otimes_{\OE} L_2,
\] 
viewed as an $\O$-module, is a self-dual lattice of $W=\V_1\otimes_F \V_2$
with respect to the integer $r_1+r_2$. In what follows, let us assume
$\V_i^a=0$ for either $i=1$ or $2$.

We would like to choose our polarization $W=W^+\oplus W^-$ for $W$ in
such a way that $A=A\cap W^+\oplus A\cap W^-$. For this, we consider
the following two cases:
\begin{enumerate}
\item[\underline{Case 1}:] $\V_2^a=0$; In this case we choose
\begin{align*}
W^+&=\V_1\otimes \V_2^+\\
W^-&=\V_1\otimes \V_2^-
\end{align*}
so that
\begin{align*}
A\cap W^+&=L_1\otimes L_2^+\\
A\cap W^-&=L_1\otimes L_2^-.
\end{align*}
Certainly we have $A=A\cap W^+\oplus A\cap W^-$.

\item[\underline{Case 2}:] $\V_1^a=0$; In this case we choose
\begin{align*}
W^+&=\V_1^+\otimes \V_2\\
W^-&=\V_1^-\otimes \V_2
\end{align*}
so that
\begin{align*}
A\cap W^+&=L_1^+\otimes L_2\\
A\cap W^-&=L_1^-\otimes L_2.
\end{align*}
Certainly we have $A=A\cap W^+\oplus A\cap W^-$.
\end{enumerate}

If $\V_1^1=\V_2^a=0$, we can choose the polarization of $W$ in either
way. For this reason, let us make the
following definition.
\begin{Def}
For $W=\V_1\otimes\V_2$, if the polarization is chosen as in Case 1
above, we call it ``Type 1 polarization''. If it is chosen as in Case 2,
we call it ``Type 2 polarization''. 
\end{Def}

\begin{Rmk}
One important thing to be noted is that if $W$ is given Type 1
polarization, the group $U(\V_1)$ is in the Siegel Levi (with respect
to this polarization), and hence $\alpha_g=0$ for all $g\in
U(\V_1)$. If $W$ is given Type 2 polarization, $U(\V_2)$ is in the
Siegel Levi and $\alpha_g=0$ for all $g\in U(\V_2)$.
\end{Rmk}

\begin{Rmk}
For a self-dual lattice $A$ in $W$, if a polarization $W=W^+\oplus W^-$ is so
chosen that $A=A\cap W^+\oplus A\cap W^-$, we say that the
polarization of $W$ is {\it compatible} with the self-dual
lattice $A$. The above discussion shows that for our
$\V_1$ and $\V_2$ with fixed self-dual lattices $L_1$ and $L_2$, both
types of polarization are compatible with
the self-dual lattice $L_1\otimes L_2$.
\end{Rmk}
\quad

Each element $w\in \V_1\otimes_F \V_2$ can be
viewed as an element in $\Hom_E(\V_1, \V_2)$ in the standard way as follows: For
$w=\sum v_1\otimes v_2$, define $w:\V_1\rightarrow \V_2$ by
\[
w(v)=\sum{\la v, v_1\ra}_1v_2
\]
for $v\in \V_1$. Also $w$ can be viewed as an element in $\Hom_E(\V_2,
\V_1)$ by
\[
w(v)=\sum{\la v, v_2\ra}_2v_1
\]
for $v\in \V_2$. Which one is meant is always clear from the context.

We define $P^+:W\rightarrow W^+\subseteq W$ to be the projection on $W^+$ and
$P^-:W\rightarrow W^-\subseteq W$ the projection on $W^-$, so we have
\[
\beta(w, w')=\la P^+(w), P^-(w')\ra,
\]
provided $\beta$ is defined with respect to the polarization
$W=W^+\oplus W^-$. Now for each $i$, define $P_i^+, P_i^a,
P_i^-:\V_i\rightarrow \V_i$ to be the projections on $\V_i^+, \V_i^a$ and $ \V_i^-$,
respectively. If we view each element $w\in
\V_1\otimes_F \V_2$ as $w\in\Hom_E(\V_1, \V_2)$ then  $P^+(w)$
and $P^-(w)$ as elements in $\Hom_E(\V_1, \V_2)$ are to be interpreted as follows.
\begin{enumerate}[$\bullet$]
\item \underline{Type 1 polarization}: $\V_2^a=0$;
\begin{align*}
P^+(w)&=P_2^+\circ w\\
P^-(w)&=P_2^-\circ w.
\end{align*}
\item \underline{Type 2 polarization}: $\V_1^a=0$;
\begin{align*}
P^+(w)&=w\circ P_1^-\\
P^-(w)&=w\circ P_1^+.
\end{align*} 
\end{enumerate}

We often write $w^+:=P^+(w)$ and $w^-=P^-(w)$. Also for $v\in\V_i$ and
we write $v^+:=P_i^+(v)$, $v^-:=P_i^-(v)$ and $v^a:=P_i^a(v)$. So for
example, if $w=v_1\otimes v_2\in \V_1\otimes\V_2$ and $W$ is given
Type 1 polarization, then $w^+=v_1\otimes v_2^+$ and $w^-=v_1\otimes
v_2^-$, and if it is given Type 2 polarization, then $w^+=v_1^+\otimes
v_2$ and $w^-=v_1^-\otimes v_2$.

\section{\bf Two key lemmas}\label{S:key_lemmas}


In this section, we will formulate the two key lemmas which would
imply the Howe duality conjecture for unramified dual pairs, if
proven. They are the analogues of Theorem I.4 and Proposition I.5 in
\cite[p.103]{MVW}. We closely follow the notations in \cite{MVW}.

As before, we fix $(\V_1,
\la\;,\;\ra_1)$ and $(\V_2, \la\;,\;\ra_2)$ together with self-dual
lattices $L_1$ and $L_2$ with respect to $r_1$ and $r_2$,
respectively, where $r_1+r_2=r$ is the exponential conductor of
$\psi$, and let
\[
W=\V_1\otimes \V_2
\]
be the symplectic space as before. We let
\[
A:=L_1\otimes L_2\subseteq W,
\]
which is a self-dual lattice of $W$ with respect to $r$.

In this section, we do not assume anything specific about the
polarization of $W$, and hence in particular, it might be neither Type 1 nor Type
2. For example if $\V_1^a\neq 0$ and $\V_2^a\neq 0$, the polarization
is neither of the two. But in this section, we include such cases in
our consideration. We always realize the Weil
representation $\omega_\psi$ of $\Spt(W)$ in the
lattice model with respect to $A$ and the chosen polarization. As
before, we denote the space of the lattice model by $S_A$ or sometimes
simply $S$. 

We let
\begin{align*}
K_i&=U(\V_i)\cap\Gamma_A\\
K_i^\circ&=U(\V_i)\cap\Gamma_A^\circ.
\end{align*}
Recall from Section \ref{S:lattice_model} that $\Gamma_A=\{g\in\Sp(W):
gA=A\}$ and $\Gamma_A^\circ=\{g\in\Gamma_A:\alpha_g(a)\in\P^r\text{
  for all $a\in A$}\}$.
Note that if the residue characteristic is odd, we always have
$K_i=K_i^\circ$ and it is
the usual maximal open compact subgroup of $U(\V_i)$. Also if the
polarization of $W$ is chosen to be Type 1 (resp. Type 2), then
$K_1=K_1^\circ$ (resp. $K_2=K_2^\circ$) even when the residue
characteristic is even.

For any lattice $L$ of $\V_1$, we let
\[
B(L)=L^\perp\otimes_{\OE} L_2\subseteq W.
\]
One can check
\[
B(L)^\perp=B(L^\perp),
\]
where the $\perp$ for $B(L)$ is with respect to $r=r_1+r_2$ and the one
for $L$ is with respect to $r_1$.
Also note that $L\subseteq L_1=L_1^\perp\subseteq L^\perp$, and
for each $u\in U(\V_1)$, we
have $(uL)^\perp=u(L^\perp)$. 

For each sublattice $L\subseteq L_1$, we define
\[
K_1(L)^\circ:=\{u\in K_1^\circ : u L^\perp\subseteq L^\perp\},
\]
where we view $K_1$ as a subgroup of $\Sp(W)$. Note that $u
L^\perp\subseteq L^\perp $ implies $u L^\perp=L^\perp$ and the condition $u
L= L$ is equivalent to $uL^\perp=L^\perp$. Hence $K_1(L^\perp)^\circ=K_1(L)^\circ$.

One can check that the group
$K_1(L)^\circ$ is an open
compact subgroup of $U(\V_1)$. Of course $K_1(L_1)^\circ=K_1^\circ$. 

Also define
\begin{align*}
J_1(L)&=\{u\in K_1(L)^\circ: (u-1)L^\perp\subseteq L\}\\
H_1(L)&=\{u\in K_1(L)^\circ: (u-1)L^{\perp}\subseteq L_1\}.
\end{align*}
Both of them are open compact subgroups of $U(\V_1)$ with
$ J_1(L)\subseteq H_1(L)$. In particular, we have
\begin{align*}
0\rightarrow J_1(L)\rightarrow
K_1(L)^\circ\rightarrow{\Aut(L^\perp/L)};\\
0\rightarrow H_1(L)\rightarrow
K_1(L)^\circ\rightarrow{\Aut(L_1/L)}.
\end{align*}
Let us note that our $J_1(L)$ and $H_1(L)$ differ from the ones in
\cite{MVW} in that we always require each element $u\in J_1(L)$ or
$H_1(L)$ be in $K_1(L)^\circ$, so that $\alpha_u(a)\in\P^r$ for all
$a\in A$, though for example if the polarization of $W$ is Type 1, we
always have $\alpha_u(a)\in\P^r$, and hence our $J_1(L)$ and $H_1(L)$
coincide with those of \cite{MVW}.

For a sublattice $M\subseteq L_2$, one defines open compact subgroups
$K_2(M)^\circ, H_2(M)$ and $J_2(M)$ of $U(\V_2)$
analogously.

\quad

We define
\[
S_{L}:=\{f\in S_A: \supp f\subseteq B(L)\}
\]
Also for any subgroup $G\subseteq U(\V_1)$ which splits in the cover
$\Spt(W)$, we define
\[
S^G:=\{f\in S_A: \omega_\psi(g)\cdot f=f \text{ for all } g\in G\}.
\]

Recall from $(\ref{E:s_w})$ that for each $w\in W$, we have defined
$s_w$ to be the unique function in $S_A$ with $\supp(s_w)=(A+w)\times
F$ such that $s_w(w,0)=1$. Then the following is an easy exercise.

\begin{Lem}
The space $S_L$ is generated by $s_w$ where $w\in B(L)$.
\end{Lem}

 We have the following lemma, which is the analogue of  part of Lemma
 in \cite[p. 102]{MVW}.
\begin{Lem}\label{L:action}
Assume $L$ is any sublattice if $L_1$. For $w\in B(L)$ and
$h\in H_1(L)$, we have the equality
\[
\omega(h)s_w=\psi(-\beta(h^{-1}w-w, w)+\alpha^h(w))s_w.
\]
In particular, the map
\[
\psi_1^w:h\mapsto \psi(-\beta(h^{-1}w-w, w)+\alpha^h(w))
\]
is a character on $H_1(L)$.
\end{Lem}
\begin{proof}
By the above lemma, the space $S_{L}$ is generated by the functions
$s_w$ for $w\in  B(L)$. Hence the first part implies the second
part. To
show the first part, let $h\in H_1(L)$. For $w'\in W$, we have
$\omega(h)s_w(w', z)=s_w(h^{-1}w', z+\alpha^h(w'))$, which is non-zero if (and only if)
$h^{-1}w'\in A+w$, \ie $w'\in A+hw$. Since $h\in H_1(L)$, we have
$hw-w\in A$, \ie $hw\in A+w$. So $w'\in A+w$. Thus the suport of
$\omega(h)s_w$ is contained in that of $s_w$. Hence $\omega(h)s_w$ is
proportional to $s_w$. To determine the constant of proportionality,
choose $w'=w$. Keeping in mind $h^{-1}w-w\in A$, we have
\begin{align*}
\omega(h)s_w(w,z)&=s_w(h^{-1}w, z+\alpha^h(w))\\
&=s_w(h^{-1}w-w+w, z+\alpha^h(w))\\
&=\psi(-\beta(h^{-1}w-w, w)+\alpha^h(w))s_w(w, z).
\end{align*}
Thus $\omega(h)s_w=\psi(-\beta(h^{-1}w-w, w)+\alpha^h(w))s_w$. 
\end{proof}

We define
\[
J_1(L)^\circ:=\bigcap_{w\in B(L)}\ker\psi_1^w.
\]
Of course by definition of $J_1(L)^\circ$, we have
\begin{Prop}\label{P:inclusion}
For any sublattice $L$ of $L_1$, 
\[
S_L\subseteq S^{J_1(L)^\circ}
\]
\end{Prop}

\begin{Rmk}
the group $J_1(L)^\circ$ is an open and compact subgroup of
$U(\V_1)$. To see this, the reader can verify that if the polarization of $W$ is Type 1,
$J_1(L)=J_1(L)^\circ$, and if the polarization of $W$ is Type 2 and
the lattice $L$ is admissible, then $J_1(L)^\circ=\{u\in J_1(L):
\alpha^u(w)\in\P^r\text{ for all $w\in B(L)$}\}$. Hence In either
case, $J_1(L)^\circ$ is an open compact subgroup of $U(\V_1)$. (To
show the latter case, use Lemma \ref{L:open_closed}.) If $W$ is of Type 2
but $L$ is not admissible, one can always find an admissible lattice
$L'\subseteq L$, and $J_1(L')^\circ\subseteq J_1(L)^\circ$, which shows
$J_1(L)^\circ$ is open and compact. 
\end{Rmk}

For a sublattice $M\subseteq L_2$
and $w\in B(M)$, we can analogously define
$\psi_2^w$ and $J_2(M)^\circ$, and have
\[
S_M\subseteq S^{J_2(M)^\circ}.
\]

Now let $\He_1$ (resp. $\He_2$) be the Hecke algebra for $\Ut(\V_1)$
(resp. $\Ut(\V_2)$) as in \cite{MVW} and $\omega(\He_2)S_L$
the subspace of $S_L$ generated by the elements of the form
$\omega(\varphi)f$ for $\varphi\in\He_2$ and $f\in S_L$. The first of
the two key lemmas for the
proof of the Howe duality is 
\begin{Conj}[First Key Lemma]\label{L:key1}
For any sublattice $L\subseteq L_1$ we have
\[
\omega(\He_2)S_L=S^{J_1(L)^\circ}.
\]
Also for each $w\in B(L)$, let
\[
M=M_w=(w(\varpi^{-r_1}L_1)+L_2)^\perp.
\]
Then
\[
\omega(\He_1)S_M=S^{J_2(M)^\circ}.
\]
\end{Conj}

Apparently this is the analogue of Theorem in \cite[p. 103]{MVW}. But
the inclusions $\omega(\He_2)S_L\subseteq S^{J_1(L)^\circ}$ and
$\omega(\He_1)S_M\subseteq S^{J_2(M)^\circ}$ immediately
follow from the previous lemma because the actions of $\He_1$ and
$\He_2$ commute. The hard part is to show the other inclusion, which
we do not know how to prove in full generality.

If $M$ is as in the above conjectural lemma, one can verify that $w\in
B(M)$. With this said, let us state the second key lemma, which is
the analogue of Proposition \cite[p.103]{MVW}. 

\begin{Conj}[Second Key Lemma]\label{L:key2}
Let $w,w'\in B(L)$ be such that
\begin{enumerate}
\item $w(\varpi^{-r_2}L_2)+L_1=w'(\varpi^{-r_2}L_2)+L_1=L^\perp$;
\item $\psi_1^w=\psi_1^{w'}$ as characters on $H_1(L)$.
\end{enumerate}
Then there exists $k\in K_1$ such that $A+w=k(A+w')$. 

Let $M=M_w$ be as in the first key lemma, and $w'\in B(M)$ be such that
\begin{enumerate}
\item $w'(\varpi^{-r_1}L_1)+L_2=M^\perp$;
\item $\psi_2^w=\psi_2^{w'}$ as characters on $H_2(M)$.
\end{enumerate}
Then there exists $k\in K_2$ such that $A+w=k(A+w')$.
\end{Conj}

Finally

\begin{Thm}
The above two conjectural lemmas imply the Howe duality conjecture in
the sense that for any irreducible admissible representation $\pi$ of
$\Ut(\V_1)$, if $\Theta_\psi(\pi)\neq 0$, then $\Theta_\psi(\pi)$ has a
unique non-zero irreducible quotient.
\end{Thm}
\begin{proof}
If we believe the above two lemmas, one can simply trace the proof for the Howe
duality conjecture for odd residual characteristic as in
\cite[p.103-106]{MVW}. We give the details for this derivation in our
special case where $L$ is of the form $L=\varpi^k L_1$ for $k\geq 1+e$
in a later section.
\end{proof}


\section{\bf A proof of a special case of the first key lemma }


In this section we prove the inclusion $S^{J_1(L)^\circ}\subseteq
\omega(\He_2)S_L$ and hence the first key lemma when $\V_2$ is
symplectic (and hence $\V_1$ is symmetric) and $L\subseteq 2\varpi
L_1$. However we need this assumption only near the end of the
section, and we do not even need to assume that one of $\V_1^a$ and
$\V_2^a$ is zero . Hence at the beginning, we do not make any assumption on
$\V_1$ and $\V_2$ except, of course, that the pair is unramified. 

Our proof is a
modification of the one given in \cite[Chapter 5.III]{MVW}. Though we try to make our
proof as self-contained as possible, the reader is aways advised to
compare ours with the one in \cite{MVW}. 

As in \cite{MVW}, we always identify $W$ with
$\Hom(\V_1, \V_2)$ or $\Hom(\V_2, \V_1)$, and which one is meant is
always clear from the context. Also we abbreviate 
\[
J=J_1(L),\quad J^\circ=J_1(L)^\circ,\quad H=H_1(L),\quad B=B(L),
\]
and the sublattice $L\subseteq L_1$ will be fixed throughout. For
$w\in W$, we define a function $s[w]: H_\beta(W)\rightarrow \C$ by 
\[
s[w](w', z)=\int_{J^\circ}\omega(u)s_w(w',z)\;du
=\int_{J^\circ}s_w(u^{-1}w',z+\alpha^u(w'))\;du.
\]
where $(w',z)\in H_\beta(W)$. (Recall from $(\ref{E:s_w})$ that $s_w$ is the unique function in
$S$ whose support is $(A+w)\times F$ and $s_w(w,0)=1$.)
\begin{Lem}
The function $s[w]$ is in the space $S^{J^\circ}$. 
\end{Lem}
\begin{proof}
To check $s[w]\in S$ it
suffices to show that
\[
s[w](a+w', z)=\psi(-\beta(a, w'))s[w](w', z)
\]
for all $a\in A$ and $(w', z)\in H_\beta(W)$. But
\begin{align*}
s[w](a+w', z)&=\int_{J^\circ}\omega(u)s_w(a+w', z)\;du\\
&=\int_{J^\circ}s_w(u^{-1}a+u^{-1}w', z+\alpha^u(a+w'))\;du\\
&=\int_{J^\circ}\psi(-\beta(u^{-1}a, u^{-1}w')+\alpha^u(a+w'))s_w(u^{-1}w',
z)\;du\\
&=\int_{J^\circ}\psi(-\beta(u^{-1}a, u^{-1}w')+\beta(u^{-1}a,
u^{-1}w')-\beta(a, w')+\alpha^u(a)+\alpha^u(w'))s_w(u^{-1}w',
z)\;du\\
&=\int_{J^\circ}\psi(-\beta(a, w'))s_w(u^{-1}w', z+\alpha^u(w'))\;du\\
&=\psi(-\beta(a, w'))\int_{J^\circ}\omega(u)s_w(w', z)\;du\\
&=\psi(-\beta(a, w'))s[w](w', z),
\end{align*}
where for the fourth equality we used
\[
\alpha^u(a+w')-\alpha^u(a)-\alpha^u(w')=\beta(u^{-1}a,
u^{-1}w')-\beta(a, w'),
\]
and for the fifth we used $\psi(\alpha^u(a))=1$.

Next we will show that $s[w]\in S^{J^\circ}$, namely
$\omega(v)s[w]=s[w]$ for $v\in J^\circ$. But
\begin{align*}
\omega(v)s[w](w',z)
&=s[w](v^{-1}w', z+\alpha^v(w'))\\
&=\int_{J^\circ}\omega(u)s_w(v^{-1}w', z+\alpha^v(w'))\;du\\
&=\int_{J^\circ}s_w(u^{-1}v^{-1}w',
z+\alpha^v(w')+\alpha^u(v^{-1}w'))\;du.
\end{align*}
Now one has
\[
\alpha^v(w')+\alpha^u(v^{-1}w')=\alpha^{vu}(w'),
\]
because
\begin{align*}
\alpha^{vu}(w')&=-\alpha_{vu}(u^{-1}v^{-1}w')\\
&=-u^{-1}\cdot \alpha_v(u^{-1}v^{-1}w')-\alpha_u(u^{-1}v^{-1}w')\\
&=-\alpha_v(v^{-1}w')-\alpha_u(u^{-1}v^{-1}w')\\
&=\alpha^v(w')+\alpha^u(v^{-1}w').
\end{align*}
Hence the integral above is written as
\[
\int_{J^\circ}s_w(u^{-1}v^{-1}w', z+\alpha^{vu}(w'))\;du,
\]
which, by the change of variable $vu\mapsto u$, becomes
\[
\int_{J^\circ}s_w(u^{-1}w', z+\alpha^{u}(w'))\;du
=\int_{J^\circ}\omega(u)s_w(w',z)\;du.
\]
Hence $\omega(v)s[w]=s[w]$.
\end{proof}

Let us note
\begin{Lem}
$s[w]\neq 0$ if and only if $s[w](w,0)\neq 0$.
\end{Lem}
\begin{proof}
Assume $s[w]\neq 0$, \ie $s[w](w', z')\neq 0$ for some $w'\in W$ and
$z'\in F$. But if $s[w](w', z')\neq 0$  for some $z'\in F$,
then $s[w](w', z)\neq 0$ for all $z\in F$. Then there exists $u\in J^\circ$ such that
$(u^{-1}w',z)$ is in the support of $s_w$, \ie $u^{-1}w'\in
A+w$. Considering $uA\subseteq A$, we have $w'\in A+uw$. So $w'=a+uw$ for
some $a\in A$. So $s[w](a+uw,z)\neq 0$. But we know
$s[w](a+uw,z)=\psi(-\beta(a, uw))s[w](uw,z)$ and $\psi(-\beta(a, uw))\neq
0$, and hence $s[w](uw,z)\neq 0$. Let $z=\alpha_u(w)$. Then
$s[w](uw,\alpha_u(w))=\omega(u^{-1})s[w](w, 0)\neq 0$. (See Remark \ref{R:important}.) But
$\omega(u^{-1})s[w] =s[w]$, and so
$s[w](w, 0)\neq 0$ . The converse is obvious.
\end{proof}

For each $w\in W$, let
\[
C(w)=\bigcup_{u\in J^\circ} u(A+w).
\]
The union here is actually finite because of the compactness of $J^\circ$.
\begin{Lem}
If $s[w]\neq 0$, the support of $s[w]$ is $C(w)\times F$.
\end{Lem}
\begin{proof}
If $(w',z)$ is in the support of $s[w]$, then the integrand in $s[w](w',z)$
is non-zero for some $u\in J^\circ$, and hence $u^{-1}w'\in A+w$,
which implies $(w',z)\in C(w)\times F$.

Now we will show that for any element $(w',z)\in C(w)$,
$s[w](w',z)\neq 0$. We may assume $z=0$, and assume
$w'=a+uw$ for some $a\in A$, considering $uA\subseteq A$. But then 
\begin{align*}
s[w](w',z)&=s[w](a+uw,0)\\
&=\psi(-\beta(a,uw))s[w](uw,0)\\
&=\psi(-\beta(a, uw)-\alpha_u(w))s[w](uw,\alpha_u(w))\\
&=\psi(-\beta(a, uw)-\alpha_u(w))\omega(u^{-1})s[w](w,0)\\
&=\psi(-\beta(a, uw)-\alpha_u(w))s[w](w,0),
\end{align*}
and by the above lemma $s[w](w,0)\neq 0$.
\end{proof}

If $s[w]\neq 0$, then it is a unique (up to constant)
function in $S^{J^\circ}$ with support equal to $C(w)\times F$,
because if $s\in S^{J^\circ}$ has support equal to $C(w)\times F$,
 its values are determined by the value at $(w,0)$. Also we have

\begin{Lem}\label{L:generated}
The space $S^{J^\circ}$ is generated by the functions of the form $s[w]$.
\end{Lem}
\begin{proof}
Let $s\in S^{J^\circ}$, and $C\times F\subseteq W\times F$ its
support. Then $C$ is a finite disjoint union of the sets of the form $A+w$ for
some $w\in W$. But since $s\in S^{J^\circ}$, we have $uC=C$ for all
$u\in J^\circ$. Hence if $A+w\subseteq C$, then $u(A+w)\subseteq C$ for
all $u\in J^\circ$. Hence $C$ is written as a finite disjoint union of
the sets of the form $C(w)$. Hence $s$ must be a linear combination of
functions whose supports are of the form $C(w)\times F$. The lemma follows.
\end{proof}

We need to be more specific about when the function $s[w]$ is
nonzero. First since the support of $s_w$ is $A+w$, the
integrand of the integral defining $s[w]$ is zero unless $u^{-1}w'\in
A+w$, and hence we have
\[
s[w](w', z)=\underset{\substack{
   u\in J^\circ \\
   u^{-1}w'\in A+w
  }}{\int}\omega(u)s_w(w',z)\;du.
\]
But as we mentioned above, the suport of $s[w]$ is $C(w)\times F$, and
hence we may assume $w'\in u'(A+w)$ for some $u'\in J^\circ$. Hence
the above integral is over the set
\[
\{u\in J^\circ:u^{-1}u'(A+w)=A+w\}.
\]
But one can see
\[
\{u\in J^\circ:u^{-1}u'(A+w)=A+w\}=\{u\in J^\circ:u(A+w)=A+w\}
\]
via the map $u\mapsto {u}^{-1} u'$. Now define a sublattice
$L_w\subseteq L_1$ by
\[
L_w^\perp=L_1+w(\varpi^{-r_2}L_2).
\]
Then $w\in B(L_w)$ and the stabilizer in $K_1$ of $A+w$ is
$H_1(L_w)$. Hence
the integral for $s[w](w', z)$ is over $J^\circ\cap H_1(L_w)$ and by
Lemma \ref{L:action}, we have
\[
s[w](w', z)=\underset{u\in J^\circ\cap H_1(L_w)}{\int}
\psi(-\beta(u^{-1}w-w,w)+\alpha^u(w))s_w(w',z)\;du.
\]
Note that the map $u\mapsto \psi(-\beta(u^{-1}w-w,w)+\alpha^u(w))$ is
a character on $H_1(L_w)$. Hence by the orthogonality of characters,
we have
\begin{Prop}\label{P:non_vanishing}
For each $w\in W$, 
\[
s[w]\neq 0\quad\text{ if and only if }\quad
\psi(-\beta(u^{-1}w-w,w)+\alpha^u(w))=1
\]
for all $u\in J^\circ$ such that $u(A+w)=A+w$. (Here the condition
$u(A+w)=A+w$ is equivalent to $uw\in A+w$.) 
\end{Prop}

\quad\\

An in \cite{MVW}, we would like to write the condition
$\psi(-\beta(u^{-1}w-w,w)+\alpha^u(w))=1$ in a more explicit form by
using Cayley transforms, whose notion we will recall now. For an
$\epsilon$-Hermitian space $(\V, \la\;,\;\ra)$ over $E$, we let
$\U(\V)$ be the Lie algebra of $U(\V)$, namely
\[
\U(\V)=\{c\in\End(\V): \la cv, v'\ra+\la v, cv'\ra=0\}. 
\]
There is a bijection 
\[
\{c\in\U(\V): 1+c\text{ is invertible}\}\cong
\{g\in U(\V): 1+g\text{ is invertible}\}
\]
where the bijection is given by
\[
g=(1-c)(1+c)^{-1},\quad c=(1-g)(1+g)^{-1}.
\]
We call $(1-c)(1+c)^{-1}$ the Cayley transform of $c$. For $x, y\in \V$,
define $c_{x, y}\in\End_E(\V)$ by
\[
c_{x, y}(v)=\la v, y\ra x-\epsilon \la v, x\ra y.
\]
One can check that $c_{x, y}\in\U(\V)$. Assume $1+c_{x,y}$ is
invertible. We let $u_{x,y}$ be the Cayley transform of $c_{x,y}$, namely
\[
u_{x,y}=(1-c_{x,y})(1+c_{x,y})^{-1}.
\]
\quad\\

Now let $x, y\in \V_1$ be given. We would like to know when
$(1+c_{x,y})^{-1}$ and hence the Cayley transform exist, and if it
does exist when it is in the group $H$ or $J$
For this purpose we need to introduce the notion of {\it order} with
respect to a lattice. Namely for an $\epsilon$-Hermitian space $(\V,
\la\;,\;\ra)$ over $E$ and a lattice $M\subseteq \V$, we define ``$M$-order''
$\ord_M:\V\rightarrow\Z$ by
\[
\ord_M(v)=\max\{m\in\Z: v\in\varpi^m M\},
\]
\ie $v\in\varpi^mM\backslash\varpi^{m+1}M$. Namely if we choose an $\OE$
basis $\{v_1,\dots,v_n\}$ of $M$ and write $v=a_1v_1+\cdots+a_nv_n$,
then $\ord_M(v)=\min\{\ord_E(a_i)\}$. (We assume $\ord_M(0)=\infty$.) Apparently
$\ord_M(x+y)\geq\min\{\ord_M(x), \ord_M(y)\}$. 

\quad

Next we need to go back to the dual pair situation where $L$ is a
sublattice of our self-dual lattice $L_1\subseteq\V_1$ and
$W=\V_1\otimes\V_2$. 

\begin{Prop}\label{P:condition}
For $x,y\in\V_1$ and $w\in W$, consider the following conditions:
\begin{enumerate}[(i)]
\item $\ord_{L_1}(x)+\ord_{L_1}(y)\geq 1-r_1$;
\item $\ord_L(x)+\ord_L(y)\geq -r_1-e$;
\item $\ord_{L}(x)+\ord_{L_1}(y)\geq -r_1-e\text{ and }
\ord_{L_1}(x)+\ord_{L}(y)\geq -r_1-e$;
\item $\ord_{L_1}(x)+\ord_{L_2}(wy)\geq -e \text{ and }
\ord_{L_2}(wx)+\ord_{L_1}(y)\geq -e$.
\end{enumerate}
Depending on which of those conditions the pair $(x,y)$ satisfies,
the Cayley transform $u_{x,y}$ satisfies the following:
\begin{enumerate}[(1)]
\item The condition $(i)$ implies $u_{x,y}$ exists and $u_{x,y}L_1\subseteq L_1$;
\item The conditions $(i)$ and $(ii)$ imply $(u_{x,y}-1)L^\perp\subseteq L$, and
  hence if the condition $\psi(\alpha_{u_{x,y}}(a))=1$ for all $a\in
  A$ is satisfied,
we have  $u_{x,y}\in J$;
\item The conditions $(i)$ and $(iii)$ imply
  $(u_{x,y}-1)L^\perp\subseteq L_1$, and hence if the condition
  $\psi(\alpha_{u_{x,y}}(a))=1$ for all $a\in A$ is satisfied,
we have  $u_{x,y}\in H$;
\item The conditions $(i)$ and $(iv)$ imply $u_{x,y}w\in A+w$.
\end{enumerate}
\end{Prop}
\begin{proof}
(1) Assume the condition $(i)$ is satisfied. Let $r_x=\ord_{L_1}(x)$ and
$r_y=\ord_{L_2}(y)$, and so $r_x+r_y\geq 1-r_1$. Then for $l\in
L_1$, we have
\begin{align*}
c_{x,y}(l)
&={\la l, y\ra}_1 x-\epsilon_1{\la l, x\ra}_1 y\\
&=\varpi^{r_x+r_y}{\la l, \varpi^{-r_y}y\ra}_1\varpi^{-r_x}x
-\epsilon_1 \varpi^{r_x+r_y}{\la l, \varpi^{-r_x}x\ra}_1\varpi^{-r_y}y,
\end{align*}
where both ${\la l, \varpi^{-r_y}y\ra}_1\varpi^{-r_x}x$ and ${\la
l, \varpi^{-r_x}x\ra}_1\varpi^{-r_y}y$ are in $\varpi^{r_1}L_1$. 
Hence
$c_{x,y}(l)\in\varpi^{r_x+r_y+r_1}L_1\subseteq\varpi L_1$. Hence
$c_{x,y}(L_1)\in\varpi L_1$ and
$1+c_{x,y}$ is invertible. (To see $1+c_{x,y}$ is invertible, notice
that for any nonzero $v\in\V_1$, we have
$\ord_{L_1}(c_{x,y}(v))>\ord_{L_1}(v)$ and so $(1+c_{x,y})(v)\neq
0$.) 

Let 
\[
u_{x,y}=(1-c_{x,y})(1+c_{x,y})^{-1}.
\]
Notice that since $c_{x,y}(L_1)\in\varpi L_1$, if
we set the topology of $\V$ to be induced from that of $E$,
the sequence $c_{x,y}^n(v)$ converges to $0$ as $n\rightarrow\infty$,
and thus we obtain the geometric series 
\[
(1+c_{x,y})^{-1}=\sum_{n=0}^\infty (-1)^nc_{x,y}^n, 
\]
which gives
\begin{equation}\label{E:geometric_series}
u_{x,y}=(1-c_{x,y}) (1+c_{x,y})^{-1}=1-2c_{x,y}+2c_{x,y}^2-2c_{x,y}^3+\cdots.
\end{equation}
Hence one can see $u_{x,y}L_1\subseteq L_1$. 

(2) Assume the conditions $(i)$ and $(ii)$. From the above series expansion of $u_{x,y}$,
one can tell that to show $(u_{x,y}-1)L^\perp\subseteq L$ it suffices to
show $2c_{x,y}L^\perp\subseteq L$. Let $r_x=\ord_{L}(x)$ and
$r_y=\ord_{L}(y)$, and so $r_x+r_y\geq -r_1$ and both
$\varpi^{-r_x}x$ and $\varpi^{-r_y}$ are in $L$ . Then for $l\in
L^\perp$, we have
\begin{align*}
2c_{x,y}(l)
&=2({\la l, y\ra}_1 x-\epsilon_1{\la l, x\ra}_1 y)\\
&=2(\varpi^{r_x+r_y}{\la l, \varpi^{-r_y}y\ra}_1\varpi^{-r_x}x
-\epsilon_1 \varpi^{r_x+r_y}{\la l, \varpi^{-r_x}x\ra}_1\varpi^{-r_y}y),
\end{align*}
where both ${\la l, \varpi^{-r_y}y\ra}_1\varpi^{-r_x}x$ and ${\la
l, \varpi^{-r_x}x\ra}_1\varpi^{-r_y}y$ are in $\varpi^{r_1}L$. 
Hence
$2c_{x,y}(l)\in\varpi^{r_x+r_y+r_1+e}L$. So if
$r_x+r_y\geq-r_1-e$, we have $2c_{x,y}(l)\in L$.

(3) This case is completely analogous to (2).

(4)
Finally assume $u_{x,y}$ exists and the condition $(iv)$ is
satisfied. We need to show $(u_{x,y}-1)w\in A$. Again by the series
expansion of $u_{x,y}$, it suffices to show $2c_{x,y}(w)\in A$. Note
that  here we are viewing $c_{x,y}$ as an operator on $W$ rather
than just on $\V_1$ in the
obvious way. Also we may assume $w=v_1\otimes v_2$ where $v_1\in V_1$
and $v_2\in V_2$. Then
\begin{align*}
2c_{x,y}(w)&=2({\la v_1, y\ra}_1 x\otimes v_2-\epsilon_1{\la v_1, x\ra}_1
y\otimes v_2)\\
&=2x\otimes {\la v_1, y\ra}_1 v_2-2\epsilon_1 y\otimes {\la v_1, x\ra}_1v_2.
\end{align*}
Note that $\ord_{L_2}( {\la v_1, y\ra}_1v_2)=\ord_{L_2}(wy)$ and
$\ord_{L_2}( {\la v_1, x\ra}_1v_2)=\ord_{L_2}(wx)$. Then 
arguing as above, one can see that the condition
$\ord_{L_1}(x)+\ord_{L_2}(wy)\geq -e$ implies
$2x\otimes wy\in A$ and the condition
$\ord_{L_2}(wx)+\ord_{L_1}(y)\geq -e$ implies $2y\otimes wy\in A$.
\end{proof}

The following special case will be also needed.

\begin{Prop}\label{P:condition2}
Assume $\epsilon_1=-1$. Let $x, y\in\V_1$ be such that $x=ay$ for $a\in E$ with
$\bar{a}= a$.  Also let $w\in W$. Consider the following conditions:
\begin{enumerate}[(i)]
\item $\ord_{L_1}(x)+\ord_{L_1}(y)\geq 1-r_1-e$;
\item $\ord_L(x)+\ord_L(y)\geq -r_1-2e$;
\item $\ord_{L}(x)+\ord_{L_1}(y)\geq -r_1-2e\text{ and }
\ord_{L_1}(x)+\ord_{L}(y)\geq -r_1-2e$;
\item $\ord_{L_1}(x)+\ord_{L_2}(wy)\geq -2e \text{ and }
\ord_{L_2}(wx)+\ord_{L_1}(y)\geq -2e$.
\end{enumerate}
Depending on which of those conditions the pair $(x,y)$ satisfies,
the Cayley transform $u_{x,y}$ satisfies the following:
\begin{enumerate}[(1)]
\item The condition $(i)$ implies $u_{x,y}$ exists and $u_{x,y}L_1\subseteq L_1$;
\item The conditions $(i)$ and $(ii)$ imply $(u_{x,y}-1)L^\perp\subseteq L$, and
  hence if the condition $\psi(\alpha_{u_{x,y}}(a))=1$ for all $a\in A$ is satisfied,
we have  $u_{x,y}\in J$;
\item The conditions $(i)$ and $(iii)$ imply
  $(u_{x,y}-1)L^\perp\subseteq L_1$, and hence if the condition
  $\psi(\alpha_{u_{x,y}}(a))=1$ for all $a\in A$ is satisfied,
we have  $u_{x,y}\in H$;
\item The conditions $(i)$ and $(iv)$ imply $u_{x,y}w\in A+w$.
\end{enumerate}
\end{Prop}
\begin{proof}
The proof is the same as the one for the previous proposition, except
that this time we have $c_{x,y}(v)=2{\la v, x\ra}_1y$ for $v\in\V_1$.
\end{proof}

\begin{Lem}\label{L:beta(u)_and_beta(c)}
Let $x, y\in\V_1$ satisfy the conditions $(i)$ and $(iv)$ of Proposition
\ref{P:condition} or Proposition \ref{P:condition2} if $x=ay$ with $\bar{a}=a$. Then
\[
\beta(u_{x,y}^{-1}w-w,w)\equiv
2\beta(c_{x,y}(w),w)\;\mod 4\P^r.
\]
\end{Lem}
\begin{proof}
For notational convenience, let $u=u_{x,y}$ and
$w'=(1-c_{x,y})^{-1}w$. Then we have
\begin{align*}
\beta(u^{-1}w-w,w)
&=\beta(u^{-1}w, w)-\beta(w, w)\\
&=\beta((1+c_{x,y})w', (1-c_{x,y})w')-\beta(w, w)\\
&=\beta((1-c_{x,y}+2c_{x,y})w', (1-c_{x,y})w')-\beta(w,w)\\
&=\beta((1-c_{x,y})w', (1-c_{x,y})w')+2\beta(c_{x,y}w',
(1-c_{x,y})w')-\beta(w,w)\\
&=\beta(w, w)+2\beta(c_{x,y}w', w')-2\beta(c_{x,y}w',
c_{x,y}w')-\beta(w, w)\\
&=2\beta(c_{x,y}w', w')-2\beta(c_{x,y}w', c_{x,y}w').
\end{align*}
Now by writing $w'=(1-c_{x,y})^{-1}w$ in terms of geometric series as
in $(\ref{E:geometric_series})$,
one can see that
\[
2\beta(c_{x,y}w', w')-2\beta(c_{x,y}w', c_{x,y}w')
=2\beta(c_{x,y}w, w)+\text{(higher terms)},
\]
where (higher terms) is the sum of the terms of the form
$4\beta(c_{x,y}^kw, c_{x,y}^lw)$ for $k\geq1$ and $l\geq 1$. Since
$(x,y)$ satisfies the condition $(iv)$, $c_{x,y}w\in A$. Hence those
higher terms are in $\P^r$, which proves the lemma.
\end{proof}

\quad

Now we can write $\beta(c_{x,y}(w), w)$ more concretely in terms of
${\la-,-\ra}_2$ as long as the polarization of $W$ is chosen to be
either Type 1 or Type 2. Accordingly, from now on we will assume
either $\V_1^a=0$ or $\V_2^a=0$, and if the former is the case, we
assume the polarization is Type 2 and the latter Type 1. If both
$\V_1^a$ and $\V_2^a$ are zero, the polarization can be take to be either Type 1
or Type 2.

\begin{Lem}\label{L:beta(c)_concrete}
$\beta(c_{x,y}(w),w)$ can be computed as follows:\\
\noindent\underline{Type 1 ($\V_2^a=0$)}:
\[
\beta(c_{x,y}(w), w)=-\tr_{E/F}\left({\la w(x), w(y)\ra}_2\right).
\]

\noindent\underline{Type 2 ($\V_1^a=0$)}:
\[
\beta(c_{x,y}(w),w)=-\tr_{E/F}\left({\la w(x^+), w(y)\ra}_2+{\la w(x),
  w(y^+)\ra}_2\right).
\]
\end{Lem}
\begin{proof}
Let $w=\sum_iv_{1,i}\otimes v_{2,i}$, where $v_{1,i}\in\V_1$ and
$v_{2,i}\in\V_2$. Then
\begin{align*}
c_{x,y}(w)&=\sum_ic_{x,y}(v_{1,i})\otimes v_{2,i}\\
&=\sum_i({\la v_{1,i}, y\ra}_1 x-\epsilon_1 {\la v_{1,i}, x\ra}_1 y)\otimes v_{2,i}\\
\end{align*}
and hence
\begin{align*}
\beta(c_{x,y}(w), w)
&=\beta(\sum_i({\la v_{1,i}, y\ra}_1 x-\epsilon_1{\la v_{1,i},
x\ra}_1 y)\otimes v_{2,i}, \sum_jv_{1,j}\otimes v_{2,j})\\
&=\sum_i\sum_j \beta(({\la v_{1,i}, y\ra}_1 x-\epsilon_1{\la v_{1,i},
x\ra}_1y)\otimes v_{2,i}, v_{1,j}\otimes v_{2,j}).\\
\end{align*}
Here
\begin{align*}
&\beta(({\la v_{1,i}, y\ra}_1 x-\epsilon_1{\la v_{1,i},
x\ra}_1 y)\otimes v_{2,i}, v_{1,j}\otimes v_{2,j})\\
=&\beta({\la v_{1,i}, y\ra}_1 x\otimes v_{2,i},
v_{1,j}\otimes v_{2,j})
-\epsilon_1 \beta({\la v_{1,i}, x\ra}_1 y \otimes v_{2,i},
v_{1,j}\otimes v_{2,j}).
\end{align*}
Now we have to argue case-by-case.

\noindent\underline{Type 1 ($\V_2^a=0$)}: Recall
$W^+=\V_1\otimes\V_2^+$ and $W^-=\V_1\otimes\V_2^-$. Then
\begin{align*}
&\beta({\la v_{1,i}, y\ra}_1 x\otimes v_{2,i}, v_{1,j}\otimes
v_{2,j})-\epsilon_1 \beta({\la v_{1,i}, x\ra}_1 y \otimes v_{2,i},
v_{1,j}\otimes v_{2,j}).\\
=&-\tr_{E/F}\left(-{\la {\la v_{1,i},
y\ra}_1 x, v_{1,j}\ra}_1\overline{\la v_{2,i}^+,v_{2,j}^-\ra}_2+\epsilon_2 {\la{\la v_{1,i},
x\ra}_1 y,v_{1,j}\ra}_1\overline{\la v_{2,i}^+,v_{2,j}^-\ra}_2\right)\\
=&-\tr_{E/F}\left(\overline{\la y, v_{1,i}\ra}_1{\la x,
  v_{1,j}\ra}_1{\la v_{2,j}^-,v_{2,i}^+\ra}_2
+\overline{\la x, v_{1,i}\ra}_1{\la y,
  v_{1,j}\ra}_1\overline{\la v_{2,i}^+,v_{2,j}^-\ra}_2\right)\\
=&-\tr_{E/F}\left({\la{\la x, v_{1,j}\ra}_1 v_{2,j}^- , {\la y,
  v_{1,i}\ra}_1 v_{2,i}^+\ra}_2
+\overline{\la {\la x, v_{1,i}\ra}_1 v_{2,i}^+, {\la y, v_{1,j}\ra}_1 v_{2,j}^-\ra}_2\right)\\
=&-\tr_{E/F}\left({\la{\la x, v_{1,j}\ra}_1 v_{2,j}^- , {\la y,
  v_{1,i}\ra}_1 v_{2,i}^+\ra}_2
+{\la {\la x, v_{1,i}\ra}_1 v_{2,i}^+, {\la y, v_{1,j}\ra}_1 v_{2,j}^-\ra}_2\right).
\end{align*}
Hence
\begin{align*}
&\sum_i\sum_j \beta(({\la v_{1,i}, y\ra}_1 x-\epsilon_1{\la v_{1,i},
x\ra}_1 y)\otimes v_{2,i}, v_{1,j}\otimes v_{2,j})\\
=&\sum_i\sum_j-\tr_{E/F}\left({\la{\la x, v_{1,j}\ra}_1 v_{2,j}^- , {\la y,
  v_{1,i}\ra}_1 v_{2,i}^+\ra}_2
+{\la {\la x, v_{1,i}\ra}_1 v_{2,i}^+, {\la y, v_{1,j}\ra}_1 v_{2,j}^-\ra}_2\right)\\
=&-\tr_{E/F}\left(\la\sum_j{\la x, v_{1,j}\ra}_1 v_{2,j}^- , \sum_i{{\la y,
  v_{1,i}\ra}_1 v_{2,i}^+\ra}_2
+\la\sum_i {\la x, v_{1,i}\ra}_1 v_{2,i}^+, \sum_j{{\la y, v_{1,j}\ra}_1
  v_{2,j}^-\ra}_2)\right)\\
=&-\tr_{E/F}\left({\la w(x)^-, w(y)^+\ra}_2+{\la w(x)^+,
  w(y)^-\ra}_2\right)\\
=&-\tr_{E/F}\left({\la w(x), w(y)\ra}_2\right)
\end{align*}

\quad

\noindent\underline{Type 2 ($\V_1^a=0$)}: Recall
$W^+=\V_1^+\otimes\V_2$ and $W^-=\V_1^-\otimes\V_2$. Then
\begin{align*}
&\beta({\la v_{1,i}, y\ra}_1 x\otimes v_{2,i}, v_{1,j}\otimes
v_{2,j})-\epsilon_1 \beta({\la v_{1,i}, x\ra}_1 y \otimes v_{2,i},
v_{1,j}\otimes v_{2,j})\\
=&-\tr_{E/F}\left({\la {\la v_{1,i}, y\ra}_1 x^+,
  v_{1,j}^-\ra}_1\overline{\la v_{2,i}, v_{2,j}\ra}_2
+\epsilon_1{\la{\la v_{1,i}, x\ra}_1y^+, v_{1,j}^-\ra}_1\overline{\la
  v_{2,i}, v_{2,j}\ra}_2\right)\\
=&-\tr_{E/F}\left(\overline{\la y, v_{1,i}\ra}_1{\la x^+,
  v_{i,j}^-\ra}_1{\la v_{2,j}, v_{2,i}\ra}_2+\overline{\la
    x,v_{1,i}\ra}_1{\la y^+, v_{1,j}^-\ra}_1\overline{\la v_{2,i},
    v_{2,j}\ra}_2\right)\\
=&-\tr_{E/F}\left({\la{\la x^+, v_{1,j}^-\ra}_1v_{2,j}, {\la y,
  v_{1,i}\ra}_1 v_{2,i}\ra}_2+\overline{\la{\la x, v_{1,i}\ra}_1 v_{2,i},
    {\la y^+,v_{1,j}^-\ra}_1v_{2,j}\ra}_2\right)\\
=&-\tr_{E/F}\left({\la{\la x^+, v_{1,j}^-\ra}_1v_{2,j}, {\la y,
  v_{1,i}\ra}_1 v_{2,i}\ra}_2+{\la{\la x, v_{1,i}\ra}_1 v_{2,i},
    {\la y^+,v_{1,j}^-\ra}_1v_{2,j}\ra}_2\right).
\end{align*}
Hence by taking $\sum_i\sum_j$ as above, one obtains the desired formula.
\end{proof}

Next we would like to know when we have $\alpha^u(w)\in\P^r$ for all
$w\in B$ when $u$ is a Cayley transform. Let us start with

\begin{Lem}\label{L:alpha^u2}
Let $w\in W$ be any.
\begin{enumerate}[(1)]
\item  Suppose $W$ is equipped with Type 1 polarization. Assume
  $u_{x,y}$ exists for $x, y\in\V_1$. Then
\[
\alpha^{u_{x,y}}(w)=0.
\]
\item Suppose $W$ is equipped with Type 2 polarization. Assume
  $u_{x,y}$ exists where $x, y\in\V_1^{\pm}$. Then
\[
\alpha^{u_{x,y}}(w)=
\begin{cases}\la c_{x,y}(w^-), w^-\ra=-2\tr_{E/F}({\la w(x),
  w(y)\ra}_2)\;&\text{if $x,y\in\V_1^+$};\\
\la w^+, c_{x,y}(w^+)\ra=2\tr_{E/F}({\la w(x),
  w(y)\ra}_2)\;&\text{if $x,y\in\V_1^-$};\\
0\;&\text{otherwise}.
\end{cases}
\]
\end{enumerate}
\end{Lem}
\begin{proof}
(1) This case is obvious because, if $W$ is given Type 1 polarization, then
$\alpha_g=0$ for any $g\in U(\V_1)$.

(2) First assume $x\in\V_1^+$ and $y\in\V_1^-$ (or $x\in\V_1^-$ and
$y\in\V_1^+$). Then one can see that $u_{x,y}$ is in the Siegel Levi
of $\Sp(W)$ and hence $\alpha^{u_{x,y}}=0$.

Next assume $x,y\in\V_1^+$, and for notational simplicity, let us
write $u=u_{x,y}$. Viewed as a operator on $W$, $c_{x,y}|_{W^+}=0$ and the image of
$c_{x,y}$ is in $W^+$. Hence $c_{x,y}^2=0$ and $c_{x,y}$ may be viewed
as a map from $W^-$ to $W^-$. Then one can see
\[
u=\begin{pmatrix}1&-2c_{x,y}\\0&1\end{pmatrix}
\quad{and}\quad
u^{-1}=\begin{pmatrix}1&2c_{x,y}\\0&1\end{pmatrix},
\]
where the matrix is chosen with respect to the polarization $W=W^++W^-$.
Now for $w=w^++w^-\in W$, 
\begin{align*}\allowdisplaybreaks
\alpha^u(w)
&=-\alpha_u(u^{-1}(w^++w^-))\\
&=-\alpha_u(w^++2c_{x,y}(w^-)+w^-)\\
&=-\frac{1}{2}\la -2c_{x,y}(w^-), w^-\ra\quad\text{by (\ref{E:alpha_g})}\\
&=\la c_{x,y}(w^-), w^-\ra.
\end{align*}
Let us write $w=\sum_iv_{1,i}\otimes v_{2,i}$. Since the
polarization is Type 2, we have $w^-=\sum_iv_{1,i}^-\otimes v_{2,i}$. Thus
\[
\la c_{x,y}(w^-), w^-\ra
=\sum_i\sum_j\la c_{x,y}(v_{1,i}^-\otimes v_{2,i}), v_{1,j}^-\otimes v_{2,j}\ra
\]
Here
\begin{align*}\allowdisplaybreaks
&\la c_{x,y}(v_{1,i}^-\otimes v_{2,j}), v_{1,i}^-\otimes v_{2,j}\ra\\
=&\tr_{E/F}\Big({\la c_{x,y}(v_{1,i}^-), v_{1,j}^-\ra}_1\overline{\la
  v_{2,i}, v_{2,j}\ra}_2\Big)\\
=&\tr_{E/F}\Big({\la \big({\la v_{1,i}^-, y\ra}_1x-\epsilon_1{\la
    v_{1,i}^-,x\ra}_1y\big), v_{1,j}^-\ra}_1\overline{\la
  v_{2,i}, v_{2,j}\ra}_2\Big)\\
=&\tr_{E/F}\Big({\la v_{1,i}^-, y\ra}_1{\la x, v_{1,j}^-\ra}_1\overline{\la
  v_{2,i}, v_{2,j}\ra}_2-\epsilon_1{\la v_{1,i}^-,x\ra}_1{\la y, v_{1,j}^-\ra}_1\overline{\la
  v_{2,i}, v_{2,j}\ra}_2\Big)\\
=&\tr_{E/F}\Big(\epsilon_1\overline{\la{\la y, v_{1,i}^-\ra}_1v_{2,i}, 
{\la x, v_{1,j}^-\ra}_1v_{2,j}\ra}_2
-\overline{\la{\la x, v_{1,i}^-\ra}_1v_{2,i}, 
{\la y, v_{1,j}^-\ra}_1v_{2,j}\ra}_2\Big).
\end{align*}
Now by taking $\sum_i\sum_j$, one can see that
\begin{align*}
\la c_{x,y}(w^-), w^-\ra
&=\tr_{E/F}\Big(\epsilon_1\overline{\la w^-(y), w^-(x)\ra}_2
-{\la w^-(x), w^-(y)\ra}_2\Big)\\
&=-\tr_{E/F}\Big({\la w^-(x), w^-(y)\ra}_2+{\la w^-(x),
  w^-(y)\ra}_2\Big)\\
&=-2 \tr_{E/F}\Big({\la w^-(x), w^-(y)\ra}_2\Big).
\end{align*}
Now since $x,y\in\V_1^+$, we have $w^-(x)=w(x)$ and
$w^-(y)=w(y)$. Thus the lemma follows.

The case $x,y\in\V_1^-$ is almost identical.
\end{proof}

This lemma immediately implies

\begin{Lem}\label{L:alpha^u1}
Let $w\in B$ and $x,y\in \V_1$. Assume $u_{x,y}$ exits.
\begin{enumerate}
\item Assume that $W$ is given Type 1 polarization. Then $\alpha^{u_{x,y}}(w)\in\P^r$;
\item Assume that $W$ is given Type 2 polarization and $x,
  y\in\V_1^{\pm}$ and $L$ is admissible. Assume further that, if $x,y\in\V_1^+$ or
  $x, y\in\V_1^-$, then the condition $(ii)$ of Proposition
  \ref{P:condition} is satisfied. Then $\alpha^{u_{x,y}}(w)\in\P^r$.
\end{enumerate}
\end{Lem}
\begin{proof}
(1) is obvious. For (2), the only case that is not too obvious is when $x,y\in\V_1^+$ or
  $x, y\in\V_1^-$. Assume $x,y\in\V_1^+$. If the condition $(ii)$ is
  satisfied, we have $c_{x,y}w^-\in L\otimes L_2$. (Note that since
  $L$ is admissible, if $w\in B$ then $w^-\in B$.) Hence
  $\alpha^{u_{x,y}}(w)=\la c_{x,y}(w^-), w^-\ra\in\P^r$. The proof for
  case $x,y\in\V_1^-$ is essentially the same.
\end{proof}

One consequence of the lemma, especially  for Type 2 polarization, is
the following.

\begin{Prop}
Let $x,y\in \V_1$ be such that the pair $(x,y)$ satisfies $(i)$ and $(ii)$ of Proposition
\ref{P:condition}. 
\begin{enumerate}
\item If $W$ is given Type 1 polarization,  then $u_{x,y}\in J^\circ$.
\item If $W$ is given Type 2 polarization and $L$ is admissible, then
  $u_{x,y}\in J^\circ$ as long as $x,y\in\V_1^\pm$.
\end{enumerate}
\end{Prop}
\begin{proof}
We prove only (2). Let $u=u_{x,y}$. Assume $W$ is given Type 2
polarization. By the above lemma, $u\in J^\circ$ if and only if
$\beta(u^{-1}w-w,w)\in\P^r$ for all $w\in B$. But this happens if and
only if $u^{-1}w-w\in L\otimes L_2$ for all $w\in B$, because $L$ is
admissible. Hence the proposition follows by Proposition
\ref{P:condition} (2). (The reader should notice that this does not necessarily follow if $L$
is not admissible.)
\end{proof}

We should mention that if $W$ is given Type 1 polarization, we
actually have $J^\circ=J$.


\quad

For $w\in W$ and $t\in\Z$ with $t\geq 0$, let us define
the condition $(wt)$ as follows:
\[
\tag{$wt$}
\varpi^{-r_1}w(2\varpi L_1\cap L)\subseteq\varpi^{-t}L_2.
\]
We define
\[
S_t:=\{s\in S: \text{ $(wt)$ holds for all }
(w,0)\in\supp(s)\}.
\]

Let us note that $S_{0}\subseteq S_{1}\subseteq S_{2}\subseteq\cdots$, and
$S=\bigcup_{t\geq 0}S_t$. The space $S_t$ is stable under the action
of $J^\circ$ because $J^\circ$ preserves the space $2\varpi L_1\cap
L$. Hence we have
\[
S^{J^\circ}=\bigcup_{t\geq 0}S_t^{J^\circ}.
\]
One can also check that if $w\in W$ satisfies the condition $(wt)$,
then $w\in S_t^{J^\circ}$.

In what follows, we will show

\begin{Prop}\label{P:step1}
Assume $\V_2$ is symplectic (if the residue characteristic of $F$ is
even), and $L$ is admissible if the polarization of $W$ is Type
2. (If the polarization of $W$ is Type 1, $L$ does not have to be admissible.)
We have the inclusion $S^{J^\circ}\subseteq\omega(\He_2)
S_{0}^{J^\circ}$.
\end{Prop}

Before going into the proof,  let us mention that this immediately implies 

\begin{Prop}\label{P:step1.5}
Under the assumption of the previous proposition, if $L\subseteq
2\varpi L_1$,  we have the equality
\[
S^{J^\circ}=\omega(\He_2)S_L,
\]
which is nothing but the first key lemma.
\end{Prop}
\begin{proof}
If  $L\subseteq 2\varpi L_1$, then $S_0=S_L$.
\end{proof}

To prove Proposition \ref{P:step1}, recall from
Proposition \ref{P:inclusion} that for each
$w\in W$ we have shown $s[w]\neq 0$ if and only if
$\psi(-\beta(u^{-1}w-w,w)+\alpha^u(w))=1$ for all $u\in J^\circ$ such
that $u^{-1}w\in A+w$.

\begin{Lem}\label{L:pair}
Let $w\in W$ be such that $s[w]\neq 0$. Also let $x,y\in
\V_1$. Further assume $x, y\in\V_1^\pm$ and $L$ is admissible in case $W$ is given Type 2
polarization. Then if the pair $(x,y)$ satisfies $(i), (ii)$ and $(iv)$ of
Proposition \ref{P:condition}. Then
\[
2{\la w(x), w(y)\ra}_2\in\PE^r.
\]
(Here $\V_2$ does not have to be symplectic.)
\end{Lem}
\begin{proof}
Recall that if the pair $(x,y)$ satisfies $(i), (ii)$ and $(iv)$ of
Proposition \ref{P:condition}, then $u_{x,y}\in J^\circ$. 

Assume $W$ is given Type 1 polarization. By Lemmas
\ref{L:beta(u)_and_beta(c)} and \ref{L:beta(c)_concrete}, one has
\[
\psi(2\tr_{E/F}{\la w(x), w(y)\ra}_2)=1.
\]
By replacing $(x,y)$ by $(ax,x)$ for any $a\in\OE^\times$, we
still have the same, namely
\[
\psi(2\tr_{E/F}{\la w(ax), w(y)\ra}_2)=\psi(2\tr_{E/F}{a\la w(x),
  w(y)\ra}_2) =1
\quad\text{for all $a\in\OE^\times$},
\]
which implies
\[
2{\la w(x), w(y)\ra}_2\in\PE^r.
\]

Assume $W$ is given Type 2 polarization. Note that
\[
{\la w(x), w(y)\ra}_2={\la w(x^+), w(y^+)\ra}_2+{\la w(x^-),
  w(y^-)\ra}_2+{\la w(x^+), w(y^-)\ra}_2+{\la w(x^-), w(y^+)\ra}_2,
\]
and so it suffices to show that $2$ times each of those terms is in $\PE^r$.

First consider the pair $(x^+, y^-)$. Since $L$ is admissible, the
pair $(x^+, y^-)$ also satisfies the conditions $(i), (ii)$ and $(iv)$
of Proposition \ref{P:condition}. By Lemma \ref{L:alpha^u1}, we
know $\alpha^{u_{x^+, y^-}}=0$. Thus by Lemmas
\ref{L:beta(u)_and_beta(c)} and \ref{L:beta(c)_concrete}, one has
\[
\psi(2\tr_{E/F}{\la w(x^+), w(y^-)\ra}_2)=1.
\]
Arguing as in the previous case, this implies 
\[
2{\la w(x^+), w(y^-)\ra}_2\in\PE^r.
\]
The case for $(x^-, y^+)$ is the same.

Next consider the pair $(x^+, y^+)$. By Lemma \ref{L:alpha^u1}, we
know $u_{x^+, y^+}\in J^\circ$. By Lemma \ref{L:alpha^u2} together with Lemmas
\ref{L:beta(u)_and_beta(c)} and \ref{L:beta(c)_concrete}, we conclude
that
\begin{align*}
\psi(-\beta(u^{-1}w-w,w)+\alpha^u(w))&=
\psi(4\tr_{E/F}{\la w(x^+), w(y^+)\ra}_2-2\tr_{E/F}{\la w(x^+),
  w(y^+)\ra}_2)\\
&=\psi(2\tr_{E/F}{\la w(x^+), w(y^+)\ra}_2)\\
&=1.
\end{align*}
Arguing as before, this implies 
\[
2{\la w(x^+), w(y^+)\ra}_2\in\PE^r.
\]
The case for $(x^-, y^-)$ is essentially the same.
\end{proof}

Using this lemma, we have

\begin{Lem}
Let $t>0$ be fixed, and let $w\in W$ be such that $s[w]\neq 0$ and
$s[w]\in S_t^{J^\circ}$, so $w$ satisfies the condition $(wt)$.
Assume $x,y\in 2\varpi L_1\cap L $ and assume further that $x,y\in
\V_1^\pm$ and $L$ is admissible if $W$ is given Type 2 polarization. Then we have
\[
{\la \varpi^{t-r_1}w(x),
  \varpi^{t-r_1}w(y)\ra}_2\in\PE^{t+r_2+1}\subseteq \PE^{r_2+2}.
\]
(Again $\V_2$ does not have to be symplectic.)
\end{Lem}
\begin{proof}
First notice that the pair $(\frac{1}{2}\varpi^{t-1-r_1}x, y)$
satisfies the conditions $(i), (ii)$
and $(iv)$. So by the above lemma, we have
$2{\la w(\frac{1}{2}\varpi^{t-1-r_1}x), w(y)\ra}_2\in\PE^r$, which gives
\[
{\la \varpi^{t-r_1}w(x), \varpi^{t-r_1}w(y)\ra}_2\in\PE^{r_2+t+1},
\]
where recall $r=r_1+r_2$. The case for Type 2 polarization can be
proven in the same way.
\end{proof}

Now we are ready to prove Proposition \ref{P:step1}. Though our proof is
almost identical to the one given in \cite[p.118-119]{MVW}, we give
details here for the sake of completeness. 

\begin{proof}[Proof of Proposition \ref{P:step1}]
We only give a proof for the case of Type 2 polarization, leaving the
Type 1 case to the reader. Hence for what follows, $L$ is admissible.

Since we have
$S^{J^\circ}=\bigcup_{t\geq 0}S_t^{J^\circ}$, it suffices to show
$S_t^{J^\circ}\subseteq \omega(\He_2) S_{t-1}^{J^\circ}$ for $t\geq
1$. From Lemma \ref{L:generated} we know that  the functions of the
form $s[w]$ where $w$ satisfies $(wt)$, \ie $\varpi^{-r_1}w(2\varpi
L_1\cap L)\subseteq\varpi^{-t}L_2$, generate the space
$S_t^{J^\circ}$. Hence we have to show that for such $s[w]$, there exists
$u\in U(\V_2)$ such that $\omega(\tilde{u})s[w]\in S_{t-1}$, where
$\tilde{u}$ is some metaplectic preimage of $u$. (Note that if
$\omega(\tilde{u})s[w]\in S_{t-1}$, then $\omega(\tilde{u})s[w]\in
S_{t-1}^{J^\circ}$ due to the commutativity of the actions of $U(\V_1)$
and $U(\V_2)$.) 

Let $\overline{X}$ be the image of $\varpi^{t-r_1}w(2\varpi L_1\cap L)$ in
$L_2/\varpi L_2$. Let $x_1,\dots,x_s\in2\varpi L_1\cap L$ be such that
the reductions of $\varpi^{t-r_1}w(x_1),\dots, \varpi^{t-r_1}w(x_s)$ in
$L_2/\varpi L_2$ form a basis of $\overline{X}$. (Here we may assume
each $x_i$ is in $(2\varpi L_1\cap L)^\pm$ for the following reason: If the
reductions of $\varpi^{t-r_1}w(x_1),\dots, \varpi^{t-r_1}w(x_s)$ form a
basis of $\overline{X}$, then the reductions of
$\varpi^{t-r_1}w(x_1^+),\dots, \varpi^{t-r_1}w(x_s^+),
\varpi^{t-r_1}w(x_1^-),\dots, \varpi^{t-r_1}w(x_s^-)$ span
$\overline{X}$, and hence can be reduced to a basis. Moreover since
$L$ and $L_1$ are admissible, $x_i^\pm\in 2\varpi L_1\cap L$.) Now by the
above lemma applied to the pair $(x_i, x_j)$, we have
\[
{\la \varpi^{t-r_1}w(x_i), \varpi^{t-r_1}w(x_j)\ra}_2\in\PE^{r_2+t+1}
\subseteq\PE^{r_2+2}
\]
for all $i,j\in\{1,\dots,s\}$. Hence by Lemma \ref{L:lattice_lin_ind}, there
exist elements $e_1,\dots, e_s\in L_2$ and subspaces $X, \V_2^\circ$ and
$Y$ of $\V_2$ such that 
\begin{itemize}
\item $\V_2=X\oplus \V_2^\circ\oplus Y$;
\item $\{e_1,\dots,e_s\}$ is a basis of $X$;
\item $X$ and $Y$ are totally isotropic;
\item $X+Y$ is orthogonal to $\V_2^\circ$;
\item $L_2=L_X\oplus L_2^\circ\oplus L_Y$, where $L_X=X\cap L_2,
  L_2^\circ=\V_2^\circ\cap L_2$ and $L_Y=Y\cap L_2$;
\item $e_i\equiv \varpi^{t-r_1}w(x_i)\mod \varpi^{t+1}L_2$ for $i\in\{1,\dots,r\}$.
\end{itemize}

Let $x\in (2\varpi L_2\cap L)^\pm$. Since $L$ and $L_1$ are admissible, $x\in
2\varpi L_1\cap L$. Let us write
\[
\varpi^{t-r_1}w(x)=v_X+v^\circ+v_Y\in L_X\oplus L_2^\circ\oplus L_Y,
\]
where $v_X\in L_X, v^\circ\in L_2^\circ$ and $v_Y\in L_Y$. Since the
reduction of $\varpi^{t-r_1}w(x_i)$ is in $\overline{X}$, the
reductions of $v^\circ$ and $v_Y$ are zero, and in particular
$v_Y\in\varpi L_2$. For each $i\in\{1,\dots,r\}$ we have
\begin{align*}
{\la e_i, v_Y\ra}_2&={\la e_i, \varpi^{t-r_1}w(x)\ra}_2\quad\text{ by the
  decomposition $L_2=L_X\oplus L_2^\circ\oplus L_Y$}\\
&={\la \varpi^{t-r_1}w(x_i)+\varpi^{t+1}l,
  \varpi^{t-r_1}w(x)\ra}_2\quad\text{ for some $l\in L_2$ }\\
&\equiv {\la \varpi^{t-r_1}w(x_i), \varpi^{t-r_1}w(x)\ra}_2\mod
\PE^{r_2+t+1}\\
&\equiv 0\mod\PE^{r_2+t+1}\quad \text{by the above lemma applied to the
  pair $(x_i, x)$}.
\end{align*}
Since $\{e_1,\dots,e_s\}$ is a basis of $L_X$ and
$L_Y\cong\Hom_{\OE}(L_X, \PE^{r_2})$ via the map $y\mapsto{\la -,
  y\ra}_2$, we obtain $v_Y\in\varpi^{t+1}L_Y$. Therefore
\[
\varpi^{t-r_1}w(2\varpi L_1\cap L)\subseteq L_X\oplus \varpi
L_2^\circ\oplus \varpi^{t+1}L_Y.
\]
Now let $w'\in C(w)$, so we may assume $w'=uw+a$ for some $a\in A$ and $u\in
J^\circ$. Then
\begin{align*}
\varpi^{t-r_1}w'(2\varpi L_1\cap L)
&=\varpi^{t-r_1}(uw+a)(2\varpi L_1\cap L)\\
&=\varpi^{t-r_1}\Big(uw(2\varpi L_1\cap L)+a(2\varpi L_1\cap L)\Big)\\
&=\varpi^{t-r_1}\Big(w(u^{-1}(2\varpi L_1\cap L))+a(2\varpi L_1\cap L)\Big)\\
&\subseteq L_X\oplus \varpi L_2^\circ\oplus\varpi^{t+1}L_Y+2\varpi^{t+1}L_2\\
&\subseteq L_X\oplus \varpi L_2^\circ\oplus\varpi^{t+1}L_Y.
\end{align*}

Let us define the operator $u\in U(\V_2)$ by
\[
u=\varpi\Id_X\oplus\Id_{\V_1^\circ}\oplus \varpi^{-1}\Id_Y,
\]
where $\Id_X, \Id_{\V_1^\circ}$ and $\Id_Y$ are the identity operators
on the corresponding spaces.  (It is clear that $u$ is indeed in
$U(\V_2)$.)  Fix a preimage $\tilde{u}$ of $u$ in $\widetilde{U}(\V_2)$.
Define $s:=\omega(\tilde{u})s[w]\in \omega(\widetilde{U}(\V_2))S_t$. We
will show $s\in S_{t-1}$, which will imply
$s[w]=\omega(\tilde{u}^{-1})s\in\omega(\He_2)S_{t-1}^{J_0}$,
and will complete the proof. For this, it
suffices to show that for all $w'\in W$ such that $s(w',0)\neq 0$ we
have $\varpi^{-r_1}w'(2\varpi L_1\cap L)\subseteq
\varpi^{t-1}L_2$. Now from $(\ref{E:action_on_lattice_model})$ one can
  see that $s(w',0)$ is (a scalar multiple of)
\[
M_{u}\circ s[w](w',0)
=\int_{A/A_u}\psi(\beta(a,w'))s[w](u^{-1}(a+w'),\alpha^u(a+w'))\;da.
\]
Hence in order for this integral to be non-zero, we must have 
\[
u^{-1}(a+w')\in C(w)\;\text{ for some $a\in A$}.
\]
But if $u^{-1}(a+w')\in C(w)$ as we have shown above, we have
\[
\varpi^{t-r_1}u^{-1}(a+w')(2\varpi L_1\cap L)
\subseteq L_X\oplus \varpi L_2^\circ\oplus\varpi^{t+1}L_Y.
\]
By multiplying $u$ to both sides, one obtains
\[
\varpi^{t-r_1}(a+w')(2\varpi L_1\cap L)
\subseteq \varpi L_X\oplus \varpi L_{t+1}^\circ\oplus\varpi L_Y=\varpi L_2,
\]
which implies
\[
\varpi^{t-r_1}w'(2\varpi L_1\cap L)
\subseteq \varpi L_2
\]
because $\varpi^{t-r_1}a(2\varpi L_1\cap L)\subseteq \varpi L_2$. This gives
$\varpi^{-r_1}w'(2\varpi L_1\cap L)\subseteq \varpi^{t-1}L_2$. The
proposition is proven.
\end{proof}

As in Proposition \ref{P:step1.5},  if $L\subseteq 2\varpi  L_1$, Proposition
\ref{P:step1} implies the desired equality
$S^{J^\circ}=\omega(\He_2)S_L$. 

\quad\\

Now, unfortunately it seems that (if the residue characteristic is
even) we cannot proceed further than this
simply by following the arguments in \cite{MVW} to show
$S^{J^\circ}=\omega(\He_2)S_L$ if $L\not\subseteq 2\varpi  L_1$ for
several reasons. For example, the analogue of Lemma II. 7 in
\cite[p.112]{MVW} cannot be shown for the case of even residual
characteristic.


\section{\bf A proof of a special case of the second key lemma}


In this section, we will give a proof of the second key lemma to the
extend necessary to prove our main theorem. Let us start with a couple
of lemmas:

\begin{Lem}\label{L:max_lattice1}
Let $L\subseteq L_1$ be a lattice. For
$w\in B(L)$, $w(\varpi^{-r_2}L_2)+L_1\neq L^\perp$ if and only if there
exists a lattice $L'$ such that
$L\subsetneq L'\subseteq L_1$ and $w\in B(L')$.
\end{Lem}

\begin{proof}
Assume $w(\varpi^{-r_2}L_2)+L_1\neq L^\perp$.
Note that it is always true
that $w(\varpi^{-r_2}L_2)+L_1\subseteq L^\perp$. Hence we assume
$w(\varpi^{-r_2}L_2)+L_1\subsetneq L^\perp$. Then we can take
$L'$ to be such that
\[
{L'}^\perp=w(\varpi^{-r_2}L_2)+L_1,
\]
so ${L'}^\perp\subsetneq L^\perp$,
which implies  $L\subsetneq {L'}$. But one can see $w\in B(L')$.

Conversely, assume there exists a lattice $L'$ such that
$L\subsetneq L'\subseteq L_1$ and $w\in B(L')$. Then
${L'}^\perp\subsetneq L^\perp$, which implies
$w(\varpi^{-r_2}L_2)+L_1\subseteq {L'}^\perp\subsetneq L^\perp$.
\end{proof}

\begin{Lem}\label{L:max_lattice2}
For a lattice $L\subseteq L_1$ and $w\in
B(L)$, $w(\varpi^{-r_2}L_2)+L_1=L^\perp$ if and only if
$w^{-1}(\varpi^{r_1}L_2)\cap L_1=L$, where
$w^{-1}(\varpi^{r_1}L_2)$ is the inverse image of
$\varpi^{r_1}L_2$ when $w$ is viewed as a map $w:\V_1\rightarrow \V_2$.
\end{Lem}
\begin{proof}
Assume $w(\varpi^{-r_2}L_2)+L_1=L^\perp$. If ${\{e_i\}}_i$ is a
basis of $L_2$, then we can write $w=\sum_im_i\otimes
e_i$, where $m_i\in L^\perp$. We then see that
\[
w(\varpi^{-r_2}L_2)+L_1=\Span{\{m_i\}}_i+L_1=L^\perp.
\]
Now let $l\in
w^{-1}(\varpi^{r_1}L_2)\cap L_1$. So $w(l)=\sum_i {\la l,
m_i\ra}_1 e_i\in \varpi^{r_1}L_2$, which means ${\la l,
m_i\ra}_1\in\PE^{r_1}$ for all $i$. Thus for any $x\in
\Span{\{m_i\}}_i+L_1= L^\perp$, we have ${\la l,
  x\ra}_1\in\PE^{r_1}$, which implies $l\in L$. The other inclusion
$L\subseteq w^{-1}(\varpi^{r_1}L_2)\cap L_1$ is clear
because $w\in B(L)$. 

Conversely, assume $w^{-1}(\varpi^{r_1}L_2)\cap
L_1=L$. But
assume $w(\varpi^{-r_2}L_2)+L_1\neq L^\perp$. By the above
lemma, there exists a lattice $L'$ such that
$L\subsetneq L'$ and $w\in B(L')$, which
implies ${L'}\subseteq
w^{-1}(\varpi^{r_1}L_2)\cap L_1$, and hence $L\subsetneq
w^{-1}(\varpi^{r_1}L_2)\cap L_1$, which is a
contradiction.
\end{proof}

Now we specialize to the orthogonal-symplectic dual pair.

\begin{Prop}\label{P:key2_for_L}
Assume the pair $(\V_1, \la-,-\ra_1)$ and $(\V_2, \la-,-\ra_2)$ is
such that $\V_2$ is symplectic. We let $\dim\V_1=m$ and
$\dim\V_2=2n$.
We give $W=\V_1\otimes\V_2$ Type 1
polarization, \ie $W^+=\V_1\otimes\V_2^+$ and
$W^-=\V_1\otimes\V_2^-$. Assume $L$ is a sublattice of
$L_1$ which admits a basis of the form
\[
\{e_1,\dots,e_s,\varpi^{t_{s+1}}e_{s+1},\dots,\varpi^{t_m}e_m\}
\]
for some $s\in\{0,\dots,m-1\}$, where $\{e_1,\dots,e_m\}$ is a basis
of $L_1$ and $t_i\geq 1+e$ for all $i\in\{r+1,\dots,m\}$.  Let $w\in
B(L)$ be such that $w(\varpi^{-r_2}L_2)+L_1=L^\perp$. Let $w'\in B(L)$ be such that 
\begin{itemize}
\item $w'(\varpi^{-r_2}L_2)+L_1=w(\varpi^{-r_2}L_2)+L_1=L^\perp$;
\item $\psi_1^{w'}=\psi_1^{w}$ as characters of $H_1(L)$.
\end{itemize}
Then there exists $k\in K_2$ such that $A+w=k(A+w')$.
\end{Prop}
\begin{proof}
Though this can be proven simply by modifying the proof of
Proposition I.5 of \cite[p.124-125]{MVW}, we will give the details.

First note that for each $i\in\{1,\dots,s\}$, we have $w(e_i),
w'(e_i)\in L_2$. Hence by adding some elements in $A$ to $w$ and $w'$,
respectively, one can assume $w(e_i)=w'(e_i)=0$ for all
$i\in\{1,\dots,r\}$. Also of course we may assume
$t_i\leq t_j$ for $i\leq j$. Let
\[
z_i=w(\varpi^{t_i-r_1}e_i)\quad\text{ and }\quad z'_i=w'(\varpi^{t_i-r_1}e_i)
\]
for each $i\in\{s+1,\dots,m\}$. Since $w,w'\in B(L)$, $z_i, z_i'\in L_2$. Moreover, the
reductions $\bar{z}_i$ in $L_2/\varpi L_2$ are linearly
independent over the residue field. To show this, consider
\[
\sum_{i=s+1}^md_iz_i\in\varpi L_2
\]
for $d_i\in\O$. Then
\[
\sum_{i=s+1}^md_i\varpi^{t_i-r_1}e_i\in w^{-1}(\varpi
L_2)\cap\varpi^{1-r_1}L_1=\varpi^{1-r_1}L,
\]
where we used the above lemme to obtain the equality. Considering the vectors
$\varpi^{t_i}e_i$ are basis vectors of $L$, we have $d_i\in\varpi\O$. Hence
the reductions $\bar{z}_i$ are linearly independent over the residue field. Similarly the
reductions $\bar{z}'_i$ are linearly independent.

Now for $i,j\in\{s+1,\dots,m\}$ with $i\leq j$ (so $t_i\leq t_j$), let $x=\frac{1}{2}e_i$ and
$y=\varpi^{t_j-r_1}e_j$. Then the pair $(x,y)$ satisfies the conditions (i) and (iii) of
Proposition \ref{P:condition}, and so $u_{x,y}\in H_1(L)$. So we have
\[
\psi_1^w(u_{x,y})=\psi_1^{w'}(u_{x,y}), 
\]
\ie
\[
\psi(-\beta(u_{x,y}^{-1}w-w,w)+\alpha^{u_{x,y}}(w))
=\psi(-\beta(u_{x,y}^{-1}w'-w',w')+\alpha^{u_{x,y}}(w')).
\]
 Also the pair $(x,y)$ satisfies the condition (iv) of Proposition
 \ref{P:condition}, so by Lemmas \ref{L:beta(u)_and_beta(c)},
 \ref{L:beta(c)_concrete} and \ref{L:alpha^u2}, this gives
\[
\psi(2{\la w(x), w(y)\ra}_2)=\psi(2{\la w'(x), w'(y)\ra}_2).
\]
By replacing $(x,y)$ by $(ax,y)$ for any $a\in\O^\times$, we still
have the same equality, and hence have
\[
2{\la w(x), w(y)\ra}_2\equiv 2{\la w'(x), w'(y)\ra}_2\mod\P^r.
\]
Thus we have
\[
{\la w(\varpi^{t_i-r_1}e_i), w(\varpi^{t_j-r_1}e_j)\ra}_2
\equiv {\la w'(\varpi^{t_i-r_1}e_i), w'(\varpi^{t_j-r_1}e_j)\ra}_2
\mod\P^{r_2+t_i},
\]
\ie
\[
{\la z_i, z_j\ra}_2\equiv {\la z_i', z_j'\ra}_2\mod\P^{r_2+t_i}.
\]
By taking $m_{ij}={\la z_i, z_j\ra}_2$ in Lemma \ref{L:lattice_symplectic},
there exist elements $z_1'',\dots,z''_m\in L_2$ such that
\begin{align*}
&z_i''-z_i'\in\P^{t_i}\quad\text{ for all $i\in\{s+1,\dots,m\}$};\\
&{\la z''_i, z''_j\ra}_2= {\la z_i, z_j\ra}_2\quad\text{ for all $i, j\in\{r+1,\dots,m\}$}.
\end{align*}
By Lemma \ref{L:lattice_Witt}, there exists $k\in K_2$ such that
$kz_i=kz_i''$ for all $i\in\{s+1,\dots,m\}$. Define $a\in A$ by
\[
a(e_i)=\begin{cases}0\quad&\text{ for all $i\in\{1,\dots,s\}$};\\
\varpi^{-t_i}(z_i''-z_i')\quad&\text{ for all $i\in\{s+1,\dots,m\}$}.
\end{cases}
\]
Indeed $a\in A$ thanks to $z_i''-z_i'\in\P^{t_i}$. One can verify that
\[
(w'+a)(e_i)=kw(e_i) \quad\text{ for all $i\in\{1,\dots,m\}$},
\]
so $w'+a=kw$. This completes the proof.
\end{proof}

\begin{Rmk}
In the above proof, the fact that the
reductions $\bar{z}_i$ are linearly independent actually
implies that $\dim \V_2\geq m-r$. Namely if $\dim \V_2< m-r$,
simply there is no $w$ with $w(L_2)+L_1=L^\perp$. So in particular, if
$L\subseteq 2\varpi L_1$, then we must that $\dim \V_2\geq\dim \V_1$.
\end{Rmk}

\begin{Prop}\label{P:key2_for_M}
Assume the pair $(\V_1, \la-,-\ra_1)$ and $(\V_2, \la-,-\ra_2)$ is
such that $\V_2$ is symplectic and $\dim\V_1=dim\V_2=2n$.
We give $W=\V_1\otimes\V_2$ Type 1
polarization, \ie $W^+=\V_1\otimes\V_2^+$ and
$W^-=\V_1\otimes\V_2^-$. Assume $L=\varpi^kL_1$ for some $k\geq 1+e$.
Let $w\in B(L)$ be such that $w(\varpi^{-r_2}L_2)+L_1=L^\perp$.
Let $M=M_w=(w(\varpi^{-r_2}L_1)+L_2)^\perp$, so $w\in B(M)$. Assume $w'\in W$ is
such that
\begin{itemize}
\item $w'(\varpi^{-r_1}L_1)+L_2=w(\varpi^{-r_1}L_1)+L_2=M^\perp$;
\item $\psi_2^{w'}=\psi_2^{w}$ as characters of $H_2(M)$.
\end{itemize}
Then there exists $k\in K_1$ such that $A+w=k(A+w')$.
\end{Prop}
\begin{proof}
Assume $\dim\V_1=\dim\V_2=2n$. Let $\{e_1^\ast,\dots,e_{2n}^\ast\}$
be a basis of $L_1$ dual to $\{e_1,\dots,e_{2n}\}$,
  \ie ${\la e_i,e_j^\ast\ra}_1=\varpi^{r_1}\delta_{ij}$ for all
  $i,j\in\{1,\dots,2n\}$, so
  $L^\perp=\Span_{\O}\{\varpi^{-k}e_{r+1}^\ast,\dots,\varpi^{-k}e_{2n}^\ast\}$. Then
  $w$ can be written as
\[
w=\sum_{i=1}^{2n}\varpi^{-k}e_i^\ast\otimes v_i,
\]
where $v_i\in L_2$. Hence
\[
w(\varpi^{-r_1}L_1)+L_2=\Span_{\O}\{\varpi^{-k}v_{1},\dots,\varpi^{-k}v_{2n}\}+L_2.
\]
We define $z_i=w(\varpi^{k}e_i)$ as in the previous
lemma. Since the reductions $\bar{z_i}$ are linearly
independent, one can see that $\{v_{1},\dots,v_{2n}\}$ can be extended
to an $\O$-basis of $L_2$ by Lemma \ref{L:basis_extension}. But
because $\dim_{\O}L_1=2n$, after all $\{v_{1},\dots,v_{2n}\}$ is an
$\O$-basis of $L_2$. Thus we have
\[
w(\varpi^{-r_1}L_1)+L_2=\Span_{\O}\{\varpi^{-k}v_{1},\dots,\varpi^{-k}v_{2n}\}.
\]
So if we let $\{v_1^\ast,\dots,v_{2n}^\ast\}$ be a basis dual to
$\{v_1,\dots,v_{2n}\}$, \ie ${\la
  v_i,v_j^\ast\ra}_2=\varpi^{r_2}\delta_{ij}$, then
\[
M=(w(\varpi^{-r_1}L_1)+L_2)^\perp
=\Span_{\O}\{\varpi^{k}v_{1}^\ast,\dots,\varpi^{k}v_{2n}^\ast\}=\varpi^kL_2.
\]
Clearly $M$ is admissible. Let $\{f_1,\dots,f_{2n}\}$ be a symplectic
basis of $\V_2$, so for each $i$, either $f_i\in L_2^+$ or $f_i\in
L_2^-$. Note that 
\[
M=\Span_{\O}\{\varpi^kf_1,\dots,\varpi^kf_{2n}\}.
\]

Now let $w'\in B(M)$ be as in the proposition and let 
\[
u_i=w(\varpi^kf_i)\quad\text{and}\quad u'_i=w'(\varpi^kf_i).
\]
Let $i, j\in\{1,\dots,2n\}$ be such that $i<j$. Also it satisfies the
condition (iv). By arguing as in
the previous proposition, we obtain
\[
{\la u_i, u_j\ra}_1\equiv {\la u_i', u_j'\ra}_1\mod\P^{r_1+k}.
\]
(Let us mention that in the previous proposition, to use Lemmas
 \ref{L:beta(c)_concrete} and \ref{L:alpha^u1}, we did not need to
 assume $L$ is admissible because the polarization of $W$ is of Type
 1. But this time, we do need $M$ to be admissible, because here we
 switch the roles of $\V_1$ and $\V_2$, \ie ``from the point of view
 of $\V_2$'', the polarization of $W$ is of Type 2. This is why we
 have to assume $L$ is of the form $\varpi^kL_1$.) Next for
$i=j\in\{1,\dots,2n\}$, let $x=\frac{1}{4}f_i^\ast$ and
$y=\varpi^{k-r_2}f_i^\ast$. The pair $(x, y)$ satisfies the
conditions (i) and (iii) of Proposition \ref{P:condition2}, and hence
$u_{x,y}\in H_2(M)$. By arguing as before, we obtain
\[
{\la u_i, u_i\ra}_1\equiv {\la u_i', u_i'\ra}_1\mod\P^{r_1+k+e}.
\]

Hence by Lemma \ref{L:lattice1}, there exist $u_1'',\dots,u_{2n}''\in
L_1$ such that 
\begin{align*}
&u_i''-u_i'\in\P^{t_i}\quad\text{ for all $i\in\{1,\dots,2n\}$};\\
&{\la u''_i, u''_j\ra}_1= {\la u_i, u_j\ra}_1\quad\text{ for all $i, j\in\{1,\dots,2n\}$}.
\end{align*}
Let $k\in K_2$ be defined by
$ku_i=ku_i''$ for all $i\in\{1,\dots,2n\}$. (Note that here unlike the
previous proposition, we do not
need Witt's extension theorem for lattices). Define $a\in A$ by
\[
a(e_i)=
\varpi^{-t_i}(u_i''-u_i')\quad\text{ for all $i\in\{1,\dots,2n\}$}.
\]
Indeed $a\in A$ thanks to $u_i''-u_i'\in\P^{t_i}$. One can verify that
\[
(w'+a)(e_i)=kw(e_i) \quad\text{ for all $i\in\{1,\dots,2n\}$},
\]
so $w'+a=kw$. This completes the proof.
\end{proof}


\section{\bf A proof of the main theorem}


We are ready to give a proof of our main theorem on the Howe
duality conjecture. Throughout this section, we assume that the
unramified pair $(\V_1, {\la-,-\ra}_1)$ and $(\V_2, {\la-,-\ra}_2)$ is such that 
\begin{itemize}
\item $\V_2$ is symplectic;
\item $\dim\V_1=\dim\V_2=2n$;
\item $W=\V_1\otimes\V_2$ is given Type 1 polarization.
\end{itemize}
Accordingly we write $U(\V_1)=\OO(2n)$ and $U(\V_2)=\Sp(2n)$. It is
well-known that the group $\OO(2n)\cdot\Sp(2n)$ splits in the
metaplectic cover $\Spt(W)$, and hence we view it as a subgroup of
$\Spt(W)$. Also we may consider the Hecke algebra $\He_1$
(resp. $\He_2$) as the one
for $U(\V_1)$ (resp. $U(\V_2)$) rather than for $\widetilde{U}(\V_1)$
(resp. $\widetilde{U}(\V_2)$). 

Let us make the following definition.
\begin{Def}
 Fix a self-dual lattice $L_1\subseteq\V_1$ with respect to a fixed
integer $r_1$. We define the conductor of an irreducible admissible
 representation $(\pi, V_\pi)$ of
 $U(\V_1)$ to be the sublattice $L$ of $L_1$ such that
\begin{itemize}
\item $V_\pi^{J_1(L)}\neq 0$;
\item $V_\pi^{J_1(L')}=0$ for all $L'\subseteq L_1$ with
  $L'\supsetneq L$.
\end{itemize}
\end{Def}

Since the groups $J_1(L)$ form a fundamental system of neighborhood of
the identity of $U(\V_1)$, every irreducible admissible representation
has a conductor. Also note that if $L=\varpi^k L_1$ for some integer
$k$, then $J_1(L)$ fits in the exact sequence
\[
0\rightarrow J_1(L)\rightarrow\operatorname{O}(2n)(\O)\rightarrow 
\operatorname{O}(2n)(\O/\varpi^{2k}\O).
\]

Then the main theorem is

\begin{Thm}
Assume $(\pi, V_\pi)$ is an irreducible admissible representation of
$U(\V_1)$ whose conductor is of the form $\varpi^kL_1$ with $k\geq 1+e$. If
$\Theta_\psi(\pi)\neq 0$, then it has a unique non-zero irreducible
quotient.
\end{Thm}

The rest of the section is devoted to a proof of this theorem. Our
proof follows the arguments in \cite[p.103-106]{MVW}. 

Let us start with the following definition: For any lattice
$L\subseteq L_1$, define
\[
\Psi(L):=\{\psi^w_1: w\in B(L)\text{ and } w(\varpi^{-r_2}L_2)+L_1=L^\perp\}.
\]
Recall that for each $w\in B(L)$, $\psi^w_1$ is a character on
$H_1(L)$. Also we define
\[
V_\pi[H_1(L), \psi_1]:=\{v\in V_\pi: \pi(h)v=\psi_1(h)v\},
\]
\ie the $H_1(L)$-isotypic component of $\pi$ of type $\psi_1$.
Then we have
\begin{Lem}\label{L:non-empty}
Let $L\subseteq 2\varpi L_1$ be a conductor of an irreducible
admissible representation $(\pi,V_\pi)$ of $\OO(2n)$. Assume
$\Theta_\psi(\pi)\neq 0$. Also let
$\Psi'(L)\subseteq\Psi(L)$ be the non-empty subset
defined by
\[
\Psi'(L)=\{\psi_1\in\Psi(L): V_\pi[H_1(L),\psi_1]\neq 0\}.
\]
Then
\[
V_\pi^{J_1(L)}=\bigoplus_{\psi_1\in\Psi'(L)}V_\pi[H_1(L), \psi_1].
\]
\end{Lem}
\begin{proof}
This is Lemma in \cite[p.104]{MVW}. But we will give a proof for the
sake of completeness. First let
\[
V'=\bigoplus_{\psi_1\in\Psi'(L)}V_\pi[H_1(L), \psi_1]
\]
We must show $V_\pi^{J_1(L)}=V'$. But the inclusion $V'\subseteq
V_\pi^{J_1(L)}$ is clear because the character $\psi_1^w$ is trivial
on $J_1(L)$ for all $w\in B(L)$

To show the inclusion $V_\pi^{J_1(L)}\subseteq V'$, it suffices to
show $V_\pi^{J_1(L)}\otimes\Theta_\psi(\pi)\subseteq
V'\otimes\Theta_\psi(\pi)$. To show this, we use the key first lemma
(Proposition \ref{P:step1.5}) and the maximality of the lattice $L$
with respect to the property $V_\pi^{J_1(L)}\neq 0$ in the following way:
Let $p:S\rightarrow\pi\otimes\Theta_\psi(\pi)$ be the
surjection. Since $V_\pi^{J_1(L)}\neq 0$, the map $p$ restricts to a
nonzero surjective map
\[
p:S^{J_1(L)}\rightarrow V_\pi^{J_1(L)}\otimes\Theta_\psi(\pi).
\]
By Proposition \ref{P:step1.5}, we have
$S^{J_1(L)}=\omega(\He_2)S_L$, so we have the surjective map
\[
p:\omega(\He_2)S_L\rightarrow V_\pi^{J_1(L)}\otimes\Theta_\psi(\pi).
\]
Hence the space $V_\pi^{J_1(L)}\otimes\Theta_\psi(\pi)$ is generated by the
vectors of the form $p(s_w)$ for $w\in B(L)$ under the action of the
Hecke algebra $\He_2$. For each $w\in B(L)$ such that $p(s_w)\neq 0$,
if we show $p(s_w)\in V_\pi[H_1(L), \psi_1^w]\otimes\Theta_\psi(\pi)$, we
will be done because this will imply $V_\pi^{J_1(L)}\otimes\Theta_\psi(\pi)\subseteq
V'\otimes\Theta_\psi(\pi)$. Assume to the contrary that $p(s_w)\in
V_\pi[H_1(L), \psi_1^w]\otimes\Theta_\psi(\pi)$. Then we have
$w(\varpi^{-r_2}L_2)+L_1\subsetneq L^\perp$. By Lemma \ref{L:max_lattice1}, there
exists a lattice $L'$ such that $L\subsetneq L'\subseteq L_1$ and
$w\in B(L')$. Then $s_w\in S_{L'}\subseteq S^{J_1(L')}$, which implies
$p(s_w)\in V_\pi^{J_1(L')}\otimes\Theta_\psi(\pi)$. But by the maximality
property of $L$, we have $V_\pi^{J_1(L')}=0$.
\end{proof}

For what follows, we fix $w\in B(L)$ to be such that
\begin{itemize}
\item $w(\varpi^{-r_2}L_2)+L_1=L^\perp$;
\item $V_\pi[H_1(L), \psi_1^w]\neq 0$, or equivalently $p(s_w)\neq 0$
  where  $p:S\rightarrow V_\pi\otimes\Theta_\psi(\pi)$ is the surjective map.
\end{itemize}
Also as we did in the previous section
\begin{itemize}
\item $M=M_w:=(w(\varpi^{-r_1}L_1)+L_2)^\perp$, so $w\in B(M)$.
\end{itemize}
Let us note that $\psi_2^w$ is a character on $H_2(M)$.
\quad

For each smooth but not necessarily irreducible representation $(\sigma_1, V_{\sigma_1})$ of
$U(\V_1)$, we define
\[
\overline{V}_{\sigma_1}:=\{v\in
V_{\sigma_1}:\sigma_1(h)v=\psi_1^w(h)v\text{ for all $h\in H_1(L)$}\}.
\]
Similarly for each smooth representation $(\sigma_2, V_{\sigma_2})$ of
$U(\V_2)$, we define
\[
\overline{V}_{\sigma_2}:=\{v\in
V_{\sigma_2}:\sigma_2(h)v=\psi_2^w(h)v\text{ for all $h\in H_2(M)$}\}.
\]

For each $i=1,2$, we let $e_i\in\He_i$ be the idempotent defined
by 
\[
e_i(u)=\begin{cases}[K_i:H_i]\psi_i^w(u)^{-1}&\text{for $u\in H_i$}\\
0&\text{for $u\in U(\V_i)$ but $u\notin H_i$},
\end{cases}
\]
where $H_1=H_1(L)$ and $H_2=H_2(M)$. Let us put 
\[
\overline{\He_i}=e_i\He_ie_i.
\]

\begin{Prop}\label{P:equality_Hecke}
Let $(\sigma_2, V_{\sigma_2})$ be a (not necessarily irreducible)
smooth non-zero representation of $U(\V_2)$, and $p:S\rightarrow
V_\pi\otimes V_{\sigma_2}$ a surjective $U(\V_1)\times
U(\V_2)$-intertwining map. Then $\overline{V}_{\sigma_2}\neq 0$ and we
have the equality
\[
\overline{V}_\pi\otimes\overline{V}_{\sigma_2}
=\pi(\overline{\He}_1)p(s_w)=\sigma_2(\overline{\He}_2)p(s_w),
\]
where $w$ is as fixed previously.
\end{Prop}

This lemma corresponds to Lemma \cite[p.104]{MVW}. But since we do not
have the $\V_2$-analogue of Proposition \ref{P:step1.5}, our proof
slightly differs from the one in \cite{MVW}. We need a couple of
lemmas to prove the proposition.

\begin{Lem}
Let $M=\varpi^kL_2$ with $k\geq 1+e$ and $N$ be a lattice with $N\subsetneq M$. Then
$H_2(N)\subsetneq H_2(M)$.
\end{Lem}
\begin{proof}
In what follows, we will construct an element in $H_2(M)$ not in
$H_2(N)$. Since $N\subsetneq M$, there exists $v\in L_2^+$ with $\ord_{L_2}(v)=0$
and $\varpi^kv\in M\backslash N$. There exists an integer $s>0$ such
that $\varpi^{k+s}v\in N$ but $\varpi^{k+s-1}v\notin N$. Then
$\varpi^{-k-s}v^\ast\in N^\perp\backslash M^\perp$, where $v^\ast$ is
the vector dual to $v$, \ie ${\la v,v^\ast\ra}_2=\varpi^{r_2}$. So
$v^\ast\in L_2^-$ and $\ord_{L_2}(v^\ast)=0$.
Let $x=\frac{\varpi^{k-r_2}}{2}v$ and $y=v^\ast$. The pair $(x,y)$ satisfies
the conditions (i) and (iii) of Proposition \ref{P:condition} (with,
of course, $L$, $L_1$ and $r_1$ replaced by $M$, $L_2$ and $r_2$, respectively), and
hence $u_{x,y}$ exists and $(u_{x,y}-1)M^\perp\subseteq L_2$. Also
since $x\in L_2^+$ and $y\in L_2^-$, one can see $u_{x,y}$ is in the
Sigel Levi of $\Sp(W)$ for our choice of polarization, so
$\alpha^{u_{x,y}}=0$. Hence $u_{x,y}\in H_2(M)$. Now
\[
u_{x,y}-1=2c_{x,y}+\text{higher terms},
\]
and
\begin{align*}
2c_{x,y}(\varpi^{-k-s}v^\ast)
&=2{\la \varpi^{-k-s}v^\ast, \frac{\varpi^{k-r_2}}{2}v\ra}_2y+2{\la
  \varpi^{-k-s}v^\ast, y\ra}_2 \frac{\varpi^{k-r_2}}{2}v\\
&={\la \varpi^{-k-s}v^\ast, \varpi^{k-r_2}v\ra}_2y\\
&=\varpi^{-s}y\notin L_2.
\end{align*}
Hence $u_{x,y}\notin H_2(N)$.
\end{proof}

\begin{Lem}\label{L:H_2_character}
Assume $M=\varpi^kL_2$. Let $w\in W$ be arbitrary. If $H_2(M)$
acts on $s_{w}$ via some character, then $w\in B(M)$ \ie $s_w\in S_M$.
\end{Lem}
\begin{proof}
 If $H_2(M)$ acts on $s_{w}$ via a character, we have 
$w\in\supp(\omega(h)s_w)$ for all $h\in H_2(M)$, \ie
 $\omega(h)s_w(w,0)\neq 0$. But $\omega(h)s_w(w,0)=s_w(h^{-1}w,\alpha^u(w))\neq 0$,
 which implies $h^{-1}w\in A+w$. So $(h-1)w\in A$. Now assume $w\notin
 B(M)$. Then $w(\varpi^{-r_1}L_1)\not\subseteq M^\perp$. Let
 $N=(M^\perp+w(\varpi^{-r_1}L_1))^\perp$, so $N\subsetneq M$ and $w\in
 B(N)$. Then
 $(h-1)N^\perp\subseteq L_2$, and so $h\in H_2(N)$ for all $h\in H_2(M)$, which would imply
 $H_2(N)=H_2(M)$. By the above lemma, we must have $N=M$. But
 $N\subsetneq M$, which is a contradiction. So $w\in B(M)$.
\end{proof}

Now we are ready to prove Proposition \ref{P:equality_Hecke}.

\begin{proof}[Proof of Proposition \ref{P:equality_Hecke}]
First we will show $\overline{V_\pi}\otimes\overline{V}_{\sigma_2}
=\pi(\overline{\He}_1)p(s_w)$. The argument is the same as in
\cite{MVW}, but we will repeat it here. 

Since $\overline{V}_\pi\subseteq V_\pi^{J_1(L)}$, and we know
$S^{J_1(L)}=\omega(\He_2)S_L$ by Proposition \ref{P:step1.5}, by
taking Lemma \ref{L:non-empty} into account, we can conclude that
$\overline{V}_\pi\otimes V_{\sigma_2}$ is generated under the action of $\He_2$ by
the elements of the form $q(s_{w'})$ with $w'\in B(M)$,
$w'(\varpi^{-r_2}L_2)+L_1=L^\perp$ and $\psi_1^{w'}=\psi_1^w$. By Proposition
\ref{P:key2_for_L}, there exists $k\in K_2$ such that $A+w=k(A+w')$,
which implies $s_w$ is proportional to $\omega(k)s_{w'}$. Thus
$\overline{V}_\pi\otimes V_{\sigma_2}$ is generated by $q(s_w)$ under
$\He_2$. The group $H_2(M)$ acts on $s_w$ via
the character $\psi_2^w$, and so $\overline{V}_{\sigma_2}\neq 0$. So
we have
\[
\overline{V}_\pi\otimes V_{\sigma_2}=\sigma_2(\He_2)p(s_w)
=\sigma_2(\He_2e_2)p(s_w),
\]
because $e_2^2=1$ and $\sigma_2(e_2)p(s_w)=p(s_w)$. Considering
$\sigma_2(e_2) V_2=\overline{V}_{\sigma_2}$, we have
\[
\overline{V}_\pi\otimes
\overline{V}_{\sigma_2}=\sigma_2(e_2)(\overline{V}_\pi\otimes
V_{\sigma_2})=\sigma_2(\overline{\He}_2)p(s_w).
\]

Next we will show the equality 
\[
\overline{V}_\pi\otimes\overline{V}_{\sigma_2}
=\pi(\overline{\He}_1)p(s_w).
\]
Let $w'\in W$ be such that $p(s_{w'})\neq 0$. (Note that viewed as a
representation of the compact group $H_2(M)$, the space $S$ of the Weil
representation can be decomposed as $S=S'\oplus\ker p$ for some
subspace $S'$, where both
$S'$ and $\ker p$ are spaces of representations of $H_2(W)$. Hence we
may assume $s_{w'}\in S'$.)
If $p(s_{w'})\in
{V}_\pi\otimes\overline{V}_{\sigma_2}$, then the group $H_2(M)$ has to
act on $p(s_{w'})$ via the character $\psi_2^w$, which implies $H_2(M)$
acts via the character $\psi_2^w$ on $s_{w'}$ because we assume
$s_{w'}\in S'$. By Lemma
\ref{L:H_2_character}, $s_{w'}\in S_M$, which implies the space $S_M$
surjects on ${V}_\pi\otimes\overline{V}_{\sigma_2}$. Hence the space
${V}_\pi\otimes\overline{V}_{\sigma_2}$ is generated under $\He_1$ by
$p(s_{w'})$ with $w'\in B(M)$ and $\psi_2^{w'}=\psi_2^w$. We will show
if $p(s_{w'})\neq 0$, then $w'(\varpi^{-r_1}L_1)+L_2=M^\perp$. Assume
$w'(\varpi^{-r_1}L_1)+L_2\neq M^\perp$, \ie $w'(\varpi^{-r_1}L_1)+L_2\subsetneq M^\perp$. Let
$L'\subseteq (w'(\varpi^{-r_2}L_2)+L_1)^\perp$. Note that $w'(\varpi^{-r_2}L_2)+L_1\subseteq
L^\perp$, \ie $L\subseteq L'$. Also for any $w_0\in W$ we have the bijection
\[
L_1/(w_0(\varpi^{-r_2}L_2)+L_1)^\perp\cong (w_0(\varpi^{-r_1}L_1)+L_2)/L_2.
\]
(See \cite[Sous-lemme, p.105]{MVW}.) Hence by choosing $w_0=w'$ we have
\[
[L_1:L']=[w'(\varpi^{-r_1}L_1)+L_2:L_2],
\]
and by choosing $w_0=w$, we have
\[
[L_1:L]=[M^\perp:L_2].
\]
But $w'(\varpi^{-r_1}L_1)+L_2\subsetneq M^\perp$ by our assumption on $w'$, so
$[L_1:L']<[L_1:L]$. Hence $L\subsetneq L'$. Now $w'\in B(L')$, so
$J_1(L')$ acts trivially on $s_{w'}$ and hence on $p(s_{w'})$. Hence
$V_\pi^{J_1(L')}\neq 0$. But by our assumption $V_\pi^{J_1(L')}= 0$ for
    any $L'$ with $L\subsetneq L'$, which is a contradiction. Therefore
    we have $w'(\varpi^{-r_1}L_1)+L_2=M^\perp$. Using Lemma \ref{P:key2_for_M} as
    above, one can see that $s_{w'}$ is a scalar multiple of
    $\omega(k)s_w$ for some $k\in K_1$. Arguing as above, one have the
    desired equality.
\end{proof}

Once Proposition \ref{P:equality_Hecke} is proven, the rest follows from the
following general fact.

\begin{Lem}\label{L:algebra}
Let $E$ be a complex vector space, and $\mathcal{A}, \mathcal{B}$ be
subalgebras of the endomorphism algebra ${\rm End}_{\C}(E)$, viewed as
a $\C$-algebra. Assume $\mathcal{A}$ and $\mathcal{B}$ commute
pointwise and there exists $e\in E$ such that
$\mathcal{A}e=\mathcal{B}e$. Then $\mathcal{A}$ is the centralizer of
$\mathcal{B}$ in ${\rm End}_{\C}(E)$ and vise versa. 
\end{Lem}
\begin{proof}
This is nothing but Lemma in p.106 of \cite{MVW}.
\end{proof}

We apply this lemma as follows. First assume $\Theta_\psi(\pi)$ has more
than two non-zero irreducible quotients, say $(\pi_2,
V_{\pi_2})$ and $(\pi'_2, V_{\pi'_2})$. Let
$(\sigma, V_{\sigma})$ be the representation of $U(\V_2)$ given by
$\sigma=\pi_2+\pi'_2$ and $V_{\sigma}=V_{\pi_2}+V_{\pi'_2}$, so we have the surjection
\[
p:S\rightarrow V_\pi\otimes V_\sigma.
\]
By Proposition \ref{P:equality_Hecke}, we have
\[
\overline{V}_\pi\otimes\overline{V}_\sigma
=\overline{V}_\pi\otimes\overline{V}_{\pi_2}+\overline{V}_{\pi_2}\otimes\overline{V}_{\pi_2'}
=\pi(\overline{\He}_1)p(s_w)
=\sigma(\overline{\He}_2)p(s_w).
\]
Let 
\[
q:\overline{V}_\pi\otimes\overline{V}_\sigma\rightarrow
\overline{V}_\pi\otimes\overline{V}_{\pi_2}
\]
be the projection on the first component. Then $q$ commutes with the
action of $\overline{\He_1}$. By applying Lemma \ref{L:algebra} with
$\mathcal{A}=\pi(\overline{\He}_1)$,
$\mathcal{B}=\sigma(\overline{\He}_2)$,
$E=\overline{V}_\pi\otimes\overline{V}_\sigma$ and $e=p(s_w)$, one can
conclude that $q\in \pi(\overline{\He}_1)$, \ie $q=\pi(\varphi)$ for
some $\varphi\in\overline{\He}_1$. But then
\[
q(\overline{V}_\pi\otimes\overline{V}_\sigma)
=\pi(\varphi) \overline{V}_\pi\otimes\overline{V}_\sigma
\neq \overline{V}_\pi\otimes\overline{V}_{\pi_2},
\]
which is a contradiction. This completes the proof of the main
theorem.


\section{\bf Lifting to smaller rank groups}


We will close up this paper with the following theorem. 

\begin{Thm}
Let $(U(\V_1), U(\V_2))$ be an unramified dual pair with $\V_2$
symplectic. Further assume $\dim\V_2<\dim\V_1$. ($\dim\V_1$ can be
even or odd.) If $(\pi, V_\pi)$ is an irreducible admissible representation of
$U(\V_1)$ with the property that
\[
V_\pi^{J_1(L)}=0
\]
 for all $L\supsetneq 2\varpi L_1$, \ie the conductor is smaller than
 $2\varpi L_1$. Then $\Theta_\psi(\pi)=0$.
\end{Thm}
\begin{proof}
Let $L$ be the conductor of $\pi$, so $L\subseteq 2\varpi L_1$. Assume
$\Theta_\psi(\pi)\neq 0$, and let $p:S\rightarrow
V_\pi\otimes\Theta_\psi(\pi)$ be the surjection. By Proposition
\ref{P:step1.5}, we know $S^{J_1(L)}=\omega(\He_2)S_L$, and hence
under $p$, the space $\omega(\He_2)S_L$ surjects on
$V_\pi^{J_1(L)}\otimes\Theta_\psi(\pi)$. So for some $w\in B(L)$, we have
$p(s_w)\neq 0$. The group $H_1(L)$ acts on $p(s_w)$ via the character
$\psi_1^w$. Hence by Lemma \ref{L:non-empty}, we must have
$w(\varpi^{-r_2}L_2)+L_1=L^\perp$. (Note that the proof of Lemma \ref{L:non-empty}
goes through without requiring $\V_1$ be even.) Then the conditions
for Proposition \ref{P:key2_for_L} are satisfied, and hence by
reasoning as in the first part of the proof of the proposition, one
can see that the images of $w(\varpi^{-r_1}L)$ in $L_2/\varpi L_2$
have to span $\dim\V_1$-dimensional space over the residue field. But
this is impossible because $\dim L_2/\varpi
L_2=\dim\V_2<\dim\V_1$. (Also see the first remaker after Proposition
\ref{P:key2_for_L}.) The theorem follows.
\end{proof}


\begin{thebibliography}{999999}

\bibitem[D]{Dieudonne} J. Dieudonn\'e, {\it La g\'eom\'etrie des groupes
    classiques}, Seconde edition, Springer-Verlag,
  Berlin-G\"ottingen-Heidelberg 1963 viii+125 pp.

\bibitem[H1]{H90} R. Howe, {\it Another look at the local
    θ-correspondence for an unramified dual pair}, Festschrift in
  honor of I. I. Piatetski-Shapiro on the occasion of his sixtieth
  birthday, Part I, 93–124, Israel
  Math. Conf. Proc., 2, Weizmann, Jerusalem, 1990.

\bibitem[H2]{H_un} R. Howe, unpublished manuscript

\bibitem[K]{Kitaoka} Y. Kitaoka, {\it Arithmetic of quadratic forms},
  Cambridge Tracts in Mathematics, 106. Cambridge University Press,
  Cambridge, 1993. x+268 pp.

\bibitem[Ku]{Kudla} S. Kudla, {\it On the local theta-correspondence},
  Invent. Math. 83 (1986), no. 2, 229–255.

\bibitem[LST]{LST} J.S. Li, B.Sun, and Y. Tian, {\it The multiplicity
    one conjecture for local theta correspondences}, Invent. Math. 184
  (2011), 117–124.

\bibitem[MVW]{MVW} C. Moeglin, M.-F. Vignéras,  and J.-L. Waldspurger,
  Jean-Loup, {\it  Correspondances de Howe sur un corps p-adique}, 
Lecture Notes in Mathematics, 1291. Springer-Verlag, Berlin, 1987.

\bibitem[O]{O'Meara} O.T. O'Meara, {\it Introduction to quadratic forms}, Springer-Verlag.

\bibitem[R]{Rao}
R. Ranga Rao, {\it
On some explicit formulas in the theory of Weil representation},
Pacific J. Math. 157 (1993), 335–371. 


\bibitem[Wa]{Wa} J.-L. Waldspurger, {\it Demonstration d'une
    conjecture de dualite de Howe dans le cas $p$p-adique, $p\neq 2$},
  Festschrift in honor of I. I. Piatetski-Shapiro on the occasion of
  his sixtieth birthday, Part I, 267-324, Israel Math. Conf. Proc., 2,
  Weizmann, Jerusalem, 1990.

\bibitem[W]{Weil} A. Weil, {\it Andre Sur certains groupes d'operateurs
    unitaires}, Acta Math. 111 (1964) 143-211.

\end{thebibliography}
\end{document}